\documentclass{amsart}
\usepackage[utf8]{inputenc}

\usepackage[utf8]{inputenc}
\usepackage[margin=1in]{geometry}
\usepackage[titletoc,title]{appendix}

\usepackage{amsmath,amsfonts,amssymb,mathtools}

\usepackage{color}
\usepackage{hyperref}
\hypersetup{colorlinks,citecolor=blue,linkcolor=red}

\usepackage{tikz}
\usetikzlibrary{decorations.markings}
\tikzset{->-/.style={decoration={
    markings,
    mark=at position #1 with {\fill (2pt,0)--(-2pt,2.31pt)--(-2pt,-2.31pt)--cycle;}},postaction={decorate}}}

\usepackage[
backend=biber,
style=alphabetic,maxnames=99,
]{biblatex}
\usepackage{amsthm}
\theoremstyle{plain}
\newtheorem{thm}{Theorem}[section] 
\newtheorem{lemma}[thm]{Lemma}
\newtheorem{cor}[thm]{Corollary}
\newtheorem{prop}[thm]{Proposition}

\theoremstyle{definition}
\newtheorem{defn}[thm]{Definition} 
\newtheorem{remark}[thm]{Remark}
\newtheorem{question}{Open Question}
\usepackage{chngcntr}
\counterwithin{figure}{section}

\newcommand{\Nhbd}{\mathcal{N}}
\newcommand{\0}{{\mathfrak{o}}}

\title{Characterizations of Stability via Morse Limit Sets}
\author{Jacob Garcia}
\date{November 2024}

\begin{document}

\begin{abstract}
    Subgroup stability is a strong notion of quasiconvexity that generalizes convex cocompactness in a variety of settings. In this paper, we characterize stability of a subgroup by properties of its limit set on the Morse boundary. Given $H<G$, both finitely generated, $H$ is stable exactly when all the limit points of $H$ are conical, or equivalently when all the limit points of $H$ are horospherical, as long as the limit set of $H$ is a compact subset of the Morse boundary for $G$. We also demonstrate an application of these results in the settings of the mapping class group for a finite type surface, $\text{Mod}(S)$. 
\end{abstract}

\maketitle


\section{Introduction}
    An important example of Kleinian groups are called convex cocompact groups. These are exactly the discrete subgroups $H<\text{Iso}^+(\mathbb{H}^3)$ whose orbit in $\mathbb{H}^3$ is convex cocompact. Additionally the quotient of $\mathbb{H}^3$ by these groups are compact Kleinian manifolds, and every infinite order element of a convex cocompact group is loxodromic. We highlight some of the other interesting properties of convex cocompact groups in the following theorem. 
    
    \begin{thm}\label{MainKleinianTheorem}(\cite{marden_1974, sullivan_1985}) 
        A Kleinian group $H<\text{Iso}^+(\mathbb{H}^3)\cong PSL_2(\mathbb{C})$ is called convex cocompact if one of the following equivalent conditions hold:
    
    \begin{enumerate}
        \item $H$ acts cocompactly on the convex hull of its limit set $\Lambda H$.

        \item Any $H$-orbit in $\mathbb{H}^3$ is quasiconvex. 
        
        \item Every limit point of $H$ is conical.

        \item $H$ acts cocompactly on $\mathbb{H}^3\cup \Omega$, where $\Omega=\partial\mathbb{H}^3\setminus \Lambda H$. \hfill $\square$
    \end{enumerate} 
    \end{thm}
    
    However other more recent versions of this relationship have been shown. Swenson showed a generalization of this theorem for Gromov hyperbolic groups equipped with their visual boundaries \cite{swenson2001}, and there has been recent interest in generalizing these relationships beyond the setting of word-hyperbolic groups. 
    For example, convex cocompact subgroups of mapping class groups acting on Teichm\"uller space, equipped with the Thurston compactification, have been characterized by Farb-Mosher \cite{Farb2001ConvexCS} as exactly the subgroups which determine Gromov hyperbolic surface group extensions. There has also been recent work done in this direction for subgroups of $Out(F_n)$, relating convex cocompact subgroups to hyperbolic extensions of free groups \cite{Dowdall2017, Hamenstdt2018, MR_and_Stab2017}. 
    
    There has also been interest in creating generalizations which are applicable for any finitely generated group. An important generalization comes from \cite{Durham_Taylor_ConvexCC_and_Stab_2015}, where Durham and Taylor introduced \textit{stability} (see Definition \ref{stabledefinition}) to characterize convex cocompact subgroups of a mapping class group in a way which is intrinsic to the geometry of the mapping class group, and in fact, generalizes the notions of convex cocompactness to any finitely generated group. A subgroup of isometries is stable when the orbit map $H\rightarrow X$ is a quasi-isometric embedding into a hyperbolic subset of $X$. The concept of stability was later generalized to \textit{strongly quasiconvex subgroup}, introduced in \cite{Tran2019}. We note that a subgroup is stable when it is undistorted and strongly quasiconvex.
    
    In the Kleinian, hyperbolic, and mapping class group settings, convex cocompactness is characterized by properties of the limit set on an appropriate boundary, as show by Kent and Leininger in \cite{kent_leininger_2008}, and independently by Hamenst\"adt in \cite{Hamenstadt2005}. For an arbitrary finitely generated group, it is possible to construct a (quasi-isometric invariant) boundary called the Morse boundary, which was introduced by Cordes in \cite{cordes_boundry2016} and expanded by Cordes and Hume in \cite{Cordes_Hume_2016_Stability}. A generalization of convex cocompactness developed by Cordes and Durham \cite{boundarycc2016}, called \textit{boundary convex cocompactness}, is an exact generalization of item (1) of Theorem \ref{MainKleinianTheorem}, in the case where $H$ is a proper action on an arbitrary proper geodesic metric space with a non-empty and compact limit set in the Morse boundary, see Definition \ref{BoundaryConvexCocompactDefinition}.

    The purpose of this paper is fully generalize item (3) of Theorem  \ref{MainKleinianTheorem} to the setting of finitely generated groups, thereby answering \cite[Question 1.15]{boundarycc2016}. In fact, we additionally generalize some other characterizations from the hyperbolic setting found in \cite{swenson2001}. We summarize the results of this paper in the following theorem:

\begin{thm}\label{MyMainTheorem}

Let $H$ be a finitely generated group acting by isometries on a proper geodesic metric space $X$. The following are equivalent:
    
    \begin{enumerate}
        \item Any $H$-orbit in $X$ is a stable embedding of $H\rightarrow X$. 

        \item $H$ acts boundary convex cocompactly on $X$.

        \item Every point in $\Lambda H$ is a conical limit point of $H$, $\Lambda H\neq\emptyset$, and $\Lambda H$ is a compact subset of the Morse boundary of $H$.

        \item Every point in $\Lambda H$ is a horospherical limit point of $H$, $\Lambda H\neq\emptyset$, and $\Lambda H$ is a compact subset of the Morse boundary of $H$.
    \end{enumerate}
    \end{thm}

\begin{remark}
    The result $(1)\Leftrightarrow (2)$ is found in the main theorem of \cite{boundarycc2016}. We show $(3)\Rightarrow (4)$ in a combination of Proposition \ref{MorseConicalImpliesMorseFunnelled} and Theorem \ref{MorseHoroLimitsAreMorseFunnelLimits}, using methods similar to \cite{swenson2001}. We show $(4)\Rightarrow (2)$ in Theorem \ref{AllMorseHoroImpliesCoboundedAction}, by first showing that non-cobounded actions on the weak convex hull of $\Lambda H$ admit a sequence of points $p_n$ which diverge quickly from the orbit (see Lemma \ref{noncoboundedactionsarewellspreadinweakconvexhull}), but then showing that the $p_n$ converge to an element of $\Lambda H$, which ultimately contradicts the conical assumption. We give an alternate proof to $(2)\Rightarrow (3)$ in Proposition \ref{BoundaryCCImpliesConical} which does not use the main theorem from \cite{boundarycc2016}.
\end{remark}

A limit point in $\Lambda H$ is \emph{conical} if the limit point is accumulated by the orbit in a strong way: every geodesic ray representing the limit point gets boundedly close to the orbit. See Definitions \ref{hyperbolicdef2} and \ref{conical limit point def}. In general, a geodesic ray which is constructed from geodesic segments $[x,hx]$ need not stay close to the orbit of $H$, even when $H$ is stably embedded in the hyperbolic setting. For an example, see \cite[Lemma 3]{swenson2001}. A limit point in $\Lambda H$ is \emph{horospherical} if it is accumulated by the orbit in a similar way: every horoball around a geodesic ray representing the limit point intersects the orbit, see Definition \ref{conical limit point def}.

We take a moment to provide a broad overview of stability in the recent literature. In addition to results for the mapping class group from above in \cite{Farb2001ConvexCS,kent_leininger_2008,Hamenstadt2005,Durham_Taylor_ConvexCC_and_Stab_2015}, it is also known that infinite index Morse subgroups of the mapping class group exactly coincide with stable subgroups \cite{Kim2019}, and stable subgroups of mapping class groups (and more generally, stable subgroups of Morse local-to-global groups) have interesting combination theorems \cite{Russell2021}.
Stability has also been studied in the context of Morse local-to-global groups \cite{CordesRussellSprianoZalloum_RegOfMorseGeoAndGrowth}, relatively hyperbolic groups \cite{MR_and_Stab2017}, and hierarchically hyperbolic groups \cite{AbbottBehrstockDurham_AcylindricalStabilityInHHG_2021,Russell2023}. It is also known that stable subgroups admit finite height \cite{Antolin2019} and that the growth series of a stable subgroup is rational \cite{Russell2021}. There has also been recent work on recognizing spaces, i.e. spaces where the orbit map induces a quasi-isometric embedding, for stable subgroups \cite{Balasubrananya2023, Zbinden2024}.

Comparing Theorem \ref{MyMainTheorem} to Theorem \ref{MainKleinianTheorem}, we see a cocompact action involving a domain of discontinuity in Theorem \ref{MainKleinianTheorem} which does not appear in Theorem \ref{MyMainTheorem}. This is because the standard methods used for showing this property rely on the fact that the (Gromov-)hyperbolic boundary for a word hyperbolic group is a compactification, and thus finding the requisite compact set needed for a cocompact action boils down to finding an appropriate closed subset. In contrast, the Morse boundary usually does not compactify the underlying group, in fact the Morse boundary compactifies a finitely generated group $H$ if and only if $H$ is word hyperbolic, see \cite[Theorem 3.10]{cordes_boundry2016} and \cite[Lemma 4.1]{boundarycc2016}. This leads to an open question:

\begin{question}
    Does there exist an appropriate classification of boundary convex cocompactness via an appropriate action on a domain of discontinuity analog? 
\end{question}

For other properties in Theorem \ref{MyMainTheorem}, we are able to address the need for some compactness in the Morse Boundary by assuming that the limit set of the group, $\Lambda H$, is compact. See Definition \ref{LimitSetDefinition} and Corollary \ref{LimitSetsInstratumAreCompact}. It is not possible to remove the compactness condition in either point (3) or (4) of Theorem \ref{MyMainTheorem}. For example, consider the group $G=\mathbb{Z}^2*\mathbb{Z}*\mathbb{Z}=\langle a,b\rangle * \langle c\rangle * \langle d\rangle$ with subgroup $H=\langle a,b,c\rangle$. As discussed in \cite[Remark 1.8]{boundarycc2016}, $H$ is isometrically embedded and convex in $G$, and so every point of $\Lambda H$ is conical with respect to $H$. In fact all rays representing a point in $\Lambda H$ travel through $H$ infinitely often. However $H$ is not hyperbolic, so $H$ is not stable. See \cite[Section 1.2]{boundarycc2016} for a complete discussion. 

\subsection{Applications}

Convex cocompact subgroups of mapping class groups have been well studied, see \cite{Farb2001ConvexCS,Hamenstadt2005}, but in particular conical limit point characterizations have been analyzed before. Let $S$ be a finite type surface, $\text{Mod}(S)$ its associated mapping class group, and let $\mathcal{T}(S)$ be its associated Teichm\"uller space. In \cite[Theorem 1.2]{kent_leininger_2008}, it is shown that a subgroup $H$ of $\text{Mod}(S)$ is convex cocompact if and only if all the limit points of $H$ in the Thurston compactification of $\mathcal{T}(S)$ are conical. A combination of Theorem \ref{MyMainTheorem}, \cite[Theorem 1.1]{Durham_Taylor_ConvexCC_and_Stab_2015}, and \cite[Theorem 1.18]{boundarycc2016} gives the following direct comparison, which uses the intrinsic geometry of $\textup{Mod}(S)$ instead of the geometry of $\mathcal{T}(S)$.

\begin{thm}\label{MappingClassTheoremIntro}
    Let $S$ be a finite type surface, and let $H<\textup{Mod}(S)$ be finitely generated. Then $H$ is a convex-cocompact subgroup of $\textup{Mod}(S)$ if and only if every point in $\Lambda H$ is a conical limit point of $H\curvearrowright\textup{Mod}(S)$, $\Lambda H\neq\emptyset$, and $\Lambda H$ is compact in the Morse boundary of $\textup{Mod}(S)$.
\end{thm}

This theorem, combined with the above result of \cite{kent_leininger_2008}, gives the following immediate corollary, which shows that conicality is a strong condition in the setting of mapping class groups:

\begin{cor}\label{TeichApplicationAllConical}
    Let $S$ be a finite type surface, and let $H <\textup{Mod}(S)$ be finitely generated. The following are equivalent:

    \begin{enumerate}
        \item Every limit point of $H$ in the Morse boundary of $\textup{Mod}(S)$ is a conical limit point of $H\curvearrowright\textup{Mod}(S)$ and $\Lambda H$ is compact.
        \item Every limit point of $H$ in the Thurston compactification of $\mathcal{T}(S)$ is a conical limit point of $H\curvearrowright\mathcal{T}(S)$.
    \end{enumerate}
\end{cor}

We also show that there exists a natural $\textup{Mod}(S)$-equivariant map from $\textup{Mod}(S)$ to $\mathcal{T}(S)$ which sends conical limit points of $H<\textup{Mod}(S)$ in the Morse boundary of $\text{Mod}(S)$ to conical limit point of $H$ in the the Thurston compactification of $\mathcal{T}(S)$. This directly proves the implication $(1)\Rightarrow(2)$ in Corollary \ref{TeichApplicationAllConical} without requiring results of \cite{kent_leininger_2008}, and in fact, does not require $H$ to be a convex cocompact subgroup. See Theorem \ref{TeichApplicationOneConical} for details.

Recall that $\text{Out}(F_n)$ denotes the group of outer automorphisms on the free group $F_n$ with $n$ generators. Hamenst\"adt and Hensel defined convex cocompact subgroups of $\text{Out}(F_n)$ as subgroups which have quasi-convex orbits on the free factor graph \cite[Definition 2]{Hamenstdt2018}. In \cite[Theorem 1.3]{Durham_Taylor_ConvexCC_and_Stab_2015}, it is shown that if $H\leq \text{Out}(F_n)$ is convex cocompact then $H$ is a stable subgroup of $\text{Out}(F_n)$. Combining this fact with Theorem \ref{MyMainTheorem}, we get the following relationship.

\begin{thm}\label{Out(Fn)result}
    Let $n\geq 3$. Suppose $H$ is a convex cocompact subgroup of $\textup{Out}(F_n)$ in the sense of \cite[Definition 2]{Hamenstdt2018}. Then every limit point of $H$ in the Morse boundary of $\textup{Out}(F_n)$ is a conical limit point of $H\curvearrowright\textup{Out}(F_n)$ and $\Lambda H$ is compact.
\end{thm}

However, in contrast of Theorem \ref{MappingClassTheoremIntro}, it is unlikely that the converse holds. Due to an announcement by Hamenst\"{a}dt \cite{HamenstadtAnnoucement}, there is a classification of stable subgroups of $\textup{Out}(F_n)$ which shows the converse of \cite[Theorem 1.3]{Durham_Taylor_ConvexCC_and_Stab_2015} does not hold. 

\subsection{Acknowledgments} I would like to thank my advisor Matthew Gentry Durham for their guidance and support during this project. Thanks to Elliott Vest for many conversations and for his comments on an earlier draft of this paper. I would like to thank Sam Taylor for a helpful conversation regarding $\text{Out}(F_n)$. I would also like to thank the referees for their helpful comments. 
    

\section{Background}
    We first begin by setting some notation and basic definitions. We recall that a metric space $X$ is \emph{proper} if closed balls are compact. A path $\alpha:I\rightarrow X$ is a \emph{geodesic} if $I\subseteq\mathbb{R}$ is a closed (potentially unbounded) interval and $\alpha$ preserves distances, i.e., if for all $s,t\in I$, $d_\mathbb{R}(s,t)=d_X(\alpha(s),\alpha(t))$. If $I=[a,b]$, we call $\alpha$ a \emph{geodesic segment}, if $I=[a,\infty)$, we call $\alpha$ a \emph{geodesic ray}, and if $I=(-\infty,\infty)$ then we call $\alpha$ a \emph{geodesic line}. Given two points $x,y\in X$, we use $[x,y]$ to denote a geodesic segment starting at $x$ and ending at $y$. If there exists a geodesic segment between any pair of points in $X$, we say $X$ is a \emph{geodesic metric space}.

    Given two geodesic segments $\alpha=[x,y]$ and $\beta=[y,z]$, we denote the (speed preserving) concatenation between then as $[x,y]*[y,z]$. Formally, given $\alpha:[0,a]\rightarrow X$ and $\beta:[0,b]\rightarrow X$ with $\alpha(a)=\beta(0)$, we have $\alpha*\beta:[0,a+b]\rightarrow X$ given by 
    \[\alpha*\beta(t)=
    \begin{cases}
        \alpha(t)&t\in[0,a],\\
        \beta(t-a)&t\in[a,a+b].\\
    \end{cases}
    \]
    We define the concatenation analogously in the case where $\alpha$ is a geodesic segment and $\beta$ is a geodesic ray.

    We use $B_K(p)$ to denote the closed ball of radius $K$ centered at $p$, i.e. $B_K(p)=\{x\in X:d(p,x)\leq K\}$. Given $A\subseteq X$ and $K\geq0$, we denote the \emph{K-neighborhood} of $A$ by $\mathcal{N}_K(A)=\{x\in X:d(x,A)\leq K\}$. Given two closed sets $A,B\subseteq X$, we denote the \emph{Hausdorff distance} between $A$ and $B$ as 
    \[d_{Haus}(A,B)=\min\{K:A\subseteq\mathcal{N}_K(B)\text{ and }B\subseteq\mathcal{N}_K(A)\}.\]
    Finally, given a closed set $A\subseteq X$, and a point $p\in X$, we denote the \emph{closest point projection} of $x$ to $A$ as
    \[\pi_A(p)=\{a\in A:d(a,p)=d(A,p)\}.\]

    We now take a moment to give the definition of a quasi-geodesic, since this term will appear frequently.

\begin{defn}
    Let $I\subseteq\mathbb{R}$ be a closed interval $X$ be a metric space, and let $\varphi:I\rightarrow X$. Let $K\geq 1$ and $C\geq0$. We call $\varphi$ a \emph{$(K,C)$-quasi-geodesic} if, for every $s,t\in I$, we have
    \[\frac{1}{K}d(s,t)-C\leq d(\varphi(s),\varphi(t))\leq Kd(s,t)+C.\]
    We call $\varphi$ a \emph{quasi-geodesic} if there exists a pair $(K,C)$ so that $\varphi$ is a $(K,C)$-quasi-geodesic.
\end{defn}

For a more thorough treatment of quasi-geodesics and their properties, we refer the reader to \cite{Clay2017-hf}.

\subsection{Hyperbolic geometry}

Here we provide a brief overview of the main result of \cite{swenson2001}, which is a direct analog of Theorem \ref{MainKleinianTheorem} in the setting (Gromov)-hyperbolic geometry. Although our main results are not in the setting of hyperbolicity, many of the tools and constructions we use are inspired by the results in this setting. We begin with the definition of a $\delta$-hyperbolic space.

\begin{defn}\label{delta hyperbolic}
    Let $X$ be a geodesic metric space. We call $X$ a \emph{$\delta$-hyperbolic} metric space if every geodesic triangle is $\delta$-slim, i.e., if for every $x,y,z\in X$, $[x,z]\subseteq\mathcal{N}_\delta([x,y]\cup[y,z])$. We call $X$ a \emph{hyperbolic space} if $X$ is $\delta$-hyperbolic for some $\delta\geq0$.
\end{defn}

One of the most useful facts in a $\delta$-hyperbolic space is that quasi-geodesics fellow-travel geodesics. This is known as the Morse lemma. A detailed proof of this lemma can be found in \cite[Theorem III.H.1.7]{bridsonhaeflingertextbook}.

\begin{lemma}[Morse Lemma]\label{MorseLemma}
    Let $X$ be a proper, geodesic $\delta$-hyperbolic space. There exists a (non-decreasing) function $N:[1,\infty)\times[0,\infty)\rightarrow[0,\infty)$ such that, for any geodesic $\alpha$ and any $(K,C)$-quasi-geodesic $\varphi:[a,b]\rightarrow X$ such that $\varphi(a),\varphi(b)\in\alpha$, we have that $\varphi\subseteq\mathcal{N}_{N(K,C)}(\alpha)$.
\end{lemma}

An important construction associated with $\delta$-hyperbolicity is the visual boundary. For more information on the visual boundary of a hyperbolic space and it's uses, we direct the reader to \cite{bridsonhaeflingertextbook} and\cite{kapovich2002boundaries}.

\begin{defn}\label{visual boundary}
    Let $X$ be a proper geodesic space, and let $\0\in X$. Let $R_\0(X)$ be the collection of all geodesic rays $\alpha:[a,\infty)\rightarrow X$ such that $\alpha(a)=\0$. Then we can define an equivalence relation on $R_\0(X)$ by setting $\alpha\sim\beta$ whenever the Hausdorff distance between $\alpha$ and $\beta$ is bounded. The \emph{visual boundary} of $X$ based at $\0$ is defined to be $\partial_\infty X_\0=R_\0(X)/\sim$. We use $\alpha(\infty)$ to refer to the equivalence class of $\alpha$ in $\partial_\infty X_\0$. We equip $\partial_\infty X_\0$ with the topology generated by the neighborhood basis for $\alpha$, \[U(\alpha,r,n)=\{\beta\in\partial_\infty X_\0:d(\alpha(t),\beta(t)\leq r\text{ for all }t\leq n\}.\]
\end{defn}

We also present another, equivalent definition for two rays to be in the same equivalence class $\alpha(\infty)$. As a note, this definition does not require either $\alpha$ or $\beta$ to be based at $\0$. 

\begin{defn}\label{asymptoticallyfellowtravel} Let $(X,d)$ be a proper, geodesic metric space, and let $\alpha:[a,\infty)\rightarrow X$ and $\beta:[b,\infty)\rightarrow X$ be two geodesic rays.
    We say $\alpha$ and $\beta$ \textit{$K$-asymptotically fellow-travel}, denoted by $\alpha \sim_K \beta$, if there exists $T\in\mathbb{R}$ so that whenever $t\geq T$, we have $d(\alpha(t),\beta(t))\leq K$.
\end{defn}

Importantly, in the context of a $\delta$-hyperbolic space, Definition \ref{asymptoticallyfellowtravel} classifies the tail-end fellow traveling distance in terms of only $\delta$, as expressed in the following lemma of Swenson, and is important for the definition of a horoball in a $\delta$-hyperbolic space.

\begin{lemma}\label{swensonfellowtravel}(\cite[Lemma 4]{swenson2001})
    Suppose $\alpha$ and $\beta$ are geodesic rays with $d_{Haus}(\alpha,\beta)<\infty$, i.e., with $\alpha(\infty)=\beta(\infty)$, then there exists an isometry $\rho:\mathbb{R}\rightarrow\mathbb{R}$ so that $\alpha\sim_{6\delta}\beta\circ\rho$.
\end{lemma}

\begin{defn}\label{hyperbolicdef1} (\cite{swenson2001})
Let $X$ be a proper, geodesic, $\delta$-hyperbolic metric space, and let $\alpha:[a,\infty)\rightarrow X$ and $\beta:[b,\infty)\rightarrow X$ be geodesic rays.
\begin{itemize}
    \item We denote the \textit{horoball about $\alpha$} by $H(\alpha)$ and define it as $H(\alpha)=\bigcup\{\beta([b,\infty)):\beta\sim_{6\delta} \alpha,~b\geq a\}$.
    \item We denote the \textit{funnel about $\alpha$} by $F(\alpha)$ and define it as $F(\alpha)=\{x\in X: d(x,\alpha)\leq d(\pi_\alpha(x),\alpha(a))\}$.
\end{itemize}
\end{defn}

\begin{remark}\label{swensonwelldefinedhoroball}
    By Lemma \ref{swensonfellowtravel}, $H(\alpha)$ is well defined.
\end{remark}

Using these definitions, we now construct what it means for a point $x\in\partial_\infty X$ to be a horospherical limit point or a funneled limit point of some subset $A\subseteq X$. Heuristically, $x$ is a horospherical limit point if every horoball around $x$ intersects $A$. The corresponding statement is true for funneled limit points. We also take a moment to define a conical limit point.

\begin{defn}\label{hyperbolicdef2}
Let $X$ be a proper, geodesic $\delta$-hyperbolic space.
\begin{itemize}
    \item Given a point $x\in\partial_\infty X_\0$ and a subset $A\subseteq X$, we say $x$ is a \textit{horospherical limit point of $A$} if, for every geodesic ray $\alpha$ with $\alpha(\infty)=x$, we have $H(\alpha)\cap A\not=\emptyset$.
    \item Given a point $x\in\partial_\infty X_\0$ and a subset $A\subseteq X$, we say $x$ is a \textit{funneled limit point of $A$} if, for every geodesic ray $\alpha$ with $\alpha(\infty)=x$, we have $F(\alpha)\cap A\not=\emptyset$.
    \item Given a point $x\in\partial_\infty X_\0$ and a subset $A\subseteq X$, we say $x$ is a \textit{conical limit point of $A$} if there exists $K>0$ such that, for every geodesic ray $\alpha$ with $\alpha(\infty)=x$, we have $\Nhbd_K(\alpha)\cap A\not=\emptyset$. 
\end{itemize}
\end{defn}

We present here for completeness a relaxed version of a claim in \cite[pg 125]{swenson2001} which shows that every horoball of a geodesic ray contains a funnel of an equivalent geodesic ray in a $\delta$-hyperbolic space.

\begin{lemma}\label{funnel_in_horoball_hyperbolic}
Let $(X,d)$ be a proper, geodesic, $\delta$-hyperbolic metric space, and let $\alpha:[0,\infty)\rightarrow X$ be a geodesic ray. Define $\alpha':[0,\infty)\rightarrow X$ by $\alpha'(t)=\alpha(t+6\delta)$. Then $F(\alpha')\subseteq H(\alpha)$.
\end{lemma}

\begin{proof}
    See Figure \ref{Figure1}. Let $p\in F(\alpha')$. We construct a geodesic ray $\beta:[b,\infty)\rightarrow X$ such that $\beta\sim_{6\delta}\alpha$ and $\beta(b)=p$: let $\beta_n$ be a geodesic segment which begins at $p$ and ends at $\alpha(n)$. Then, after potentially passing to a subsequence, the $\beta_n$ converge to a geodesic ray $\beta$ by the Arzelà–Ascoli theorem. Then as shown in \cite[pg 427-428]{bridsonhaeflingertextbook}, $\beta\sim_{6\delta}\alpha$. 
    
    Let $q\in \pi_{\alpha'}(p)$ such that $d(\alpha(0),q)=\min\{d(\alpha(0),x):x\in\pi_{\alpha'}(p)\}$, i.e., so that $q$ is the point in $\pi_{\alpha'}(p)$ closest to $\alpha(0)$. Notice that since $p\in F(\alpha')$ we have that $d(p,q)\leq d(q,\alpha'(0))$. Choose $T\geq 6\delta$ so that $q\in [\alpha'(0),\alpha'(T)]$ and so that for all $t\geq T$, we have $d(\alpha(t),\beta(t))<6\delta$. Then
    \begin{align*}
    T-b=
    d(\beta(T),p) &\leq d(\beta(T),\alpha(T))+d(\alpha(T),q)+d(q,p)\\ &\leq 6\delta + d(\alpha'(T-6\delta),q) + d(q,\alpha'(0))=6\delta+(T-6\delta).
    \end{align*}
    This shows that $b\geq0$, and so $p=\beta(b)\in H(\alpha)$.
\end{proof}

\begin{figure}[h]
\centering
\begin{tikzpicture}
        \draw [->, black] (0,0) to (10,0);
        \draw node[label=below:{$\alpha$}] at (7,0) {};
        \draw node[label=left:{$\alpha(0)$},fill,circle,scale=0.5] at (0,0) {};
        \draw node[label=right:{$p=\beta(b)$},fill,circle,scale=0.5] at (5,3) {};
        \draw node[fill,circle,scale=0.5] at (1,0) {};
        \draw node[label=below:{$\alpha(6\delta)$}] at (1,-0.2) {};
        \draw node[label=below:{$=\alpha'(0)$},scale=0.5] at (1,-0.7) {};
        \draw node[label=below:{$\alpha(T)$},fill,circle,scale=0.5] at (9,0) {};
        \draw node[label=below:{$=\alpha'(T-6\delta)$},scale=0.5] at (9,-0.5) {};
        \draw node[label=above:{$\beta(T)$},fill,circle,scale=0.5] at (9,0.5) {};
        \draw [->, black]   (5,3) to[out=-85,in=180] (10,0.5);
        \draw [-, black,dashed]   (-1,0) to[out=90,in=180] (5,5);
        \draw [-, black,dashed]   (-1,0) to[out=-90,in=117] (0,-3);
        \draw node[label=above:{$\beta$}] at (7,0.9) {};
        \draw [-, black] (5,3) to (4,0);
        \draw [-, black] (5,3) to (5,0);
        \draw [-, blue, ultra thick] (4,0) to (5,0);
        \draw node[label=below:{\textcolor{blue}{$\pi_{\alpha'}(p)$}}, text=blue] at (5,0) {};
        \draw node[label=below:{$q$},fill,circle,scale=0.5] at (4,0) {};
        \draw [-, black, dashed] (6,5) to (1,0);
        \draw node[label=right:{$F(\alpha')$},scale=0.5] at (6,4.5) {};
        \draw node[label=right:{$H(\alpha)$},scale=0.5] at (3,4.5) {};
        \draw [-, black, dashed] (4,-3) to (1,0);
        \draw [-, black] (9,0.5) to (9,0);
        \draw node[label=right:{$\leq6\delta$}] at (9,0.25) {};
        \end{tikzpicture}\\
        \caption{Diagram for Lemma \ref{funnel_in_horoball_hyperbolic}}\label{Figure1}

\end{figure}

We also include the complementary statement that every funnel of a geodesic ray contains a horoball of an equivalent geodesic ray.

\begin{lemma}\label{horoball_in_funnel_hyperbolic}
\textit{(\cite[Lemma 5]{swenson2001})} Let $(X,d)$ be a proper, geodesic, $\delta$-hyperbolic metric space, and let $\alpha:[0,\infty)\rightarrow X$ be a geodesic ray. Define $\alpha':[0,\infty)\rightarrow X$ by $\alpha'(t)=\alpha(t+12\delta)$. Then $H(\alpha')\subseteq F(\alpha)$. \hfill $\square$
\end{lemma}

The combination of Lemmas \ref{funnel_in_horoball_hyperbolic} and \ref{horoball_in_funnel_hyperbolic} give the following relationship, which was originally stated as a corollary in \cite{swenson2001}.

\begin{cor}\label{limit_points_hyperbolic}
In a proper, geodesic, $\delta$-hyperbolic metric space, the funneled limit points are exactly the horospherical limit points. \hfill $\square$
\end{cor}

\subsection{Morse Boundary and Morse Rays}

The fact that $X$ is $\delta$-hyperbolic is an important part of the definition of a horoball in Definition \ref{hyperbolicdef1}, as we note in Remark \ref{swensonwelldefinedhoroball}. Since our main goal of Theorem \ref{MyMainTheorem} does not have $X$ as a $\delta$-hyperbolic space, we will need to develop some analog to Lemma \ref{swensonfellowtravel} which does not use hyperbolicity. Or strategy for creating such an analog will be to use properties of Morse rays. We begin by recalling the definition.

\begin{defn} \label{MorseRayDefinition}
\textit{(\cite[Definition 1.3]{cordes_boundry2016})} A (quasi)-geodesic $\gamma$ in a metric space is called \textit{$N$-Morse}, where $N$ is a function $[1,\infty)\times[0,\infty)\rightarrow[0,\infty)$, if for any $(K,C)$-quasi-geodesic $\varphi$ with endpoints on $\gamma$, we have $\varphi\subset \Nhbd_{N(K,C)}(\gamma)$. We call the function $N$ a \textit{Morse gauge}. We say $\gamma$ is \textit{Morse} if there exists a Morse gauge $N$ so that $\gamma$ is $N$-Morse.
\end{defn}

Comparing this definition with Lemma \ref{MorseLemma} shows that Morse rays are the rays in $X$ which have hyperbolic like properties. In the next definition, Cordes uses the visual boundary (see Definition \ref{visual boundary} above) to construct a boundary on proper, geodesic $X$ without requiring $X$ to be hyperbolic.

\begin{defn}\textit{(\cite{cordes_boundry2016, Cordes_Hume_2016_Stability})} Given a Morse gauge $N$ and a basepoint $\0\in X$, the \textit{$N$-Morse stratum}, denoted $X_{\0}^N$, is defined as the set of all points $x$ such that $[\0,x]$ is an $N$-Morse geodesic. Each such stratum is $\delta$-hyperbolic for $\delta$ depending only on $N$ \cite[Proposition 3.2]{Cordes_Hume_2016_Stability}, and thus has a well defined visual boundary $\partial_\infty X^N_{\0}$. If $\mathcal{M}$ is the set of all Morse gauges, then there is a natural partial order on $\mathcal{M}$: $N\leq N'$ if $N(K,C)\leq N'(K,C)$ for all $K$ and $C$. Note the natural inclusion $\partial_\infty X_{\0}^N \hookrightarrow \partial_\infty X_{\0}^{N'}$ is continuous whenever $N\leq N'$ by \cite[Corollary 3.2]{cordes_boundry2016}. We define the \textit{Morse boundary based at $\0$} as
$$ \partial X_{\0} = \varinjlim_\mathcal{M} \partial_\infty X_{\0}^N$$
with the induced direct limit topology. Given a Morse geodesic ray $\alpha$, we denote the associated point in $\partial X_{\0}$ as $\alpha(\infty)$.
\end{defn}

\begin{remark}
Often when studying the Morse boundary, the basepoint is suppressed from the notation, as the Morse boundary is basepoint independent \cite[Proposition 2.5]{cordes_boundry2016}. However, we will often make use of the basepoint explicitly in the arguments to come, thus we keep it in the notation.
\end{remark}

\begin{remark}
    When $X$ is a $\delta$-hyperbolic space, $\partial_\infty X_\0 = \partial X_\0$. This is because, by the Morse Lemma (Lemma \ref{MorseLemma}) there exists a maximum Morse gauge $N$ so that $X=X_\0^N$. See \cite{cordes_boundry2016} for details.
\end{remark}

The following fact states that subrays of Morse rays are also Morse. This will be especially useful in Section 3, as many of the arguments which describe the relationships between horoballs, funnels, and cones require restriction to a subray, as illustrated in the proof of Lemma \ref{funnel_in_horoball_hyperbolic}.

\begin{lemma}\label{MorseSubraysAreMorse} (\cite[Lemma 3.1]{liu_2021})
    Let $X$ be a geodesic metric space. Let $\alpha:I\rightarrow X$ be an $N$-Morse $(\lambda,\epsilon)$-quasi-geodesic where $I$ is an interval of $\mathbb{R}$. Then for any interval $I'\subseteq I$, the $(\lambda,\epsilon)$-quasi-geodesic $\alpha'=\alpha|_{I'}$ is $N'$-Morse where $N'$ depends only on $\lambda,\epsilon$, and $N$. \hfill $\square$
\end{lemma}

We now present a combination of statements which will show that, given one Morse ray and another ray which fellow-travels with the first, then eventually the fellow-travelling constant is determined only by the Morse gauge of the first ray. We begin by recalling two relevant facts from \cite{cordes_boundry2016}.

\begin{prop}\label{thatonecordesprop} \textit{(\cite[Proposition 2.4]{cordes_boundry2016})}
Let $X$ be a geodesic metric space. Let $\alpha:[0,\infty)\rightarrow X$ be an $N$-Morse geodesic ray. Let $\beta:[0,\infty)\rightarrow$ be a geodesic ray such that $d(\alpha(t),\beta(t))<K$ for $t\in[A,A+D]$ for some $A\in[0,\infty)$ and $D\geq 6K$. Then for all $t\in[A+2K,A+D-2K]$, \[\hfill d(\alpha(t),\beta(t))<4N(1,2N(5,0))+2N(5,0)+d(\alpha(0),\beta(0)).\eqno\qed\] 
\end{prop}

\begin{cor}\label{cordescornumber2} \textit{(\cite[Corollary 2.6]{cordes_boundry2016})}
Let $X$ be a geodesic metric space. Let $\alpha:[0,\infty)\rightarrow X$ be an $N$-Morse geodesic ray. Let $\beta:[0,\infty)\rightarrow$ be a geodesic ray such that $d(\alpha(t),\beta(t))<K$ for all $t\in[0,\infty)$ (i.e. $\beta(\infty)=\alpha(\infty)$). Then for all $t\in[2K,\infty)$, \[d(\alpha(t),\beta(t))<\max\{4N(1,2N(5,0))+2N(5,0),8N(3,0)\}+d(\alpha(0),\beta(0)).\eqno\qed\]
\end{cor}

The proof of Proposition \ref{thatonecordesprop}, as presented in \cite{cordes_boundry2016}, shows the following additional fact:

\begin{cor}\label{cordescornumber1} Let $X$ be a geodesic metric space. Let $\alpha:[0,\infty)\rightarrow X$ be an $N$-Morse geodesic ray. Let $\beta:[0,\infty)\rightarrow$ be a geodesic ray such that $d(\alpha(t),\beta(t))<K$ for $t\in[A,A+D]$ for some $A\in[0,\infty)$ and $D\geq 6K$. Then there exists $x,y\in[0,A+2K]$ such that $d(\alpha(y),\beta(x))<N(5,0)$.
\end{cor}

\begin{proof}
    The value $x$ from the first paragraph of the proof in \cite[Proposition 2.4]{cordes_boundry2016} satisfies $x\in[\max\{0,A-2K\},A+2K]$, and as $[\max\{0,A-2K\},A+2K]\subseteq[0,A+2K]$, we get $x\in[0,A+2K]$. The third to last paragraph defines $y$ so that $\alpha(y)\in\pi_\alpha(\beta(x))$, and shows $y\leq A+2K$ and that $d(\alpha(y),\beta(x))<N(5,0)$.
\end{proof}

Corollaries \ref{cordescornumber1} and \ref{cordescornumber2} combine to give the following generalization of \cite[Chapter 3, Lemma 3.3]{bridsonhaeflingertextbook}.

\begin{prop}\label{betterprop2.4}
Let $X$ be a geodesic metric space. Let $\alpha:[0,\infty)\rightarrow X$ be an $N$-Morse geodesic ray. Let $\beta:[0,\infty)\rightarrow X$ be a geodesic ray such that $d(\alpha(t),\beta(t))<K$ for all $t\in[0,\infty)$ (i.e. $\beta(\infty)=\alpha(\infty)$). Then there exists $T_1,T_2>0$ such that for all $t\in[0,\infty)$, \[d(\alpha(T_1+t),\beta(T_2+t))<\max\{4N(1,2N(5,0))+2N(5,0),8N(3,0)\}+N(5,0).\eqno\qed\]
\end{prop}

\begin{proof}
    By Corollary \ref{cordescornumber1}, there exists $x,y\geq0$ so that $d(\alpha(x),\beta(y))<N(5,0)$. Define $\alpha'(t)=\alpha(x+t)$ and $\beta'(t)=\beta(y+t)$, and note in particular that $\alpha'(0)=\alpha(x)$ and $\beta'(0)=\beta(y)$. Applying Corollary \ref{cordescornumber2} to $\alpha'$ and $\beta'$ produces the desired result.
\end{proof}
For convenience, we will denote \[\delta_N=\max\{4N(1,2N(5,0))+2N(5,0),8N(3,0)\}+N(5,0).\] 

Using this notation, Proposition \ref{betterprop2.4} leads to the following generalization of \cite[Lemma 4]{swenson2001}.

\begin{cor}\label{betterswensonlemma4}
Let $X$ be a geodesic metric space. Let $\alpha:[0,\infty)\rightarrow X$ be an $N$-Morse geodesic ray. Let $\beta:[0,\infty)\rightarrow X$ be a geodesic ray such that $\beta(\infty)=\alpha(\infty)$. Then there exists $a\in\mathbb{R}$ and an isometry $\rho:[a,\infty)\rightarrow[0,\infty)$ so that $\alpha\sim_{\delta_N}\beta\circ\rho$.
\end{cor}
\begin{proof}
    Apply Proposition \ref{betterprop2.4} to find $T_1,T_2>0$ so that for all $t\in[0,\infty)$, $d(\alpha(T_1+t),\beta(T_2+t))<\delta_N$. Then let $\rho:[a,\infty)\rightarrow[0,\infty)$ be the unique isometry such that $\rho(T_1)=T_2$.
\end{proof}

\begin{prop}\label{ParallelRaysWithSameStartareMorse}
    Suppose $\alpha:[a,\infty)\rightarrow X$ is an $N$-Morse geodesic ray and $\beta:[b,\infty)\rightarrow X$ is a geodesic ray such that $\beta \sim_{\delta_N}\alpha$ and $\alpha(a)=\beta(b)$. Then $\beta$ is $M$-Morse where $M$ depends only on $N$.
\end{prop}

\begin{proof}
    It suffices to show that $d_{Haus}(\alpha,\beta)\leq K$ where $K\geq0$ depends only on $N$. Choose $T>0$ so that $d(\alpha(t),\beta(t))\leq \delta_N$ for all $t\geq T$. Note that $[\beta(b),\beta(t)]*[\beta(t),\alpha(t)]$ is a $(1,2\delta_N)$ quasi-geodesic, so by \cite[Lemma 2.1]{cordes_boundry2016}, $d_{Haus}([\alpha(a),\alpha(t)],[\beta(b),\beta(t)]*[\beta(t),\alpha(t)])\leq L$ for some $L$ depending only on $N$. But since $d(\alpha(t),\beta(t))\leq\delta_N$, we have $d_{Haus}([\alpha(a),\alpha(t)],[\beta(b),\beta(t)])\leq L+\delta_N$.
\end{proof}

The above statement leads to the following generalization, which is very similar to \cite[Lemma 2.8]{cordes_boundry2016}. This statement will be useful for showing a generalization of Corollary \ref{limit_points_hyperbolic}, since our horoballs and funnels will be restricted to a single Morse stratum. See Theorem \ref{MorseHoroLimitsAreMorseFunnelLimits}.

\begin{prop}\label{ParallelRaysAreMorse}
    Suppose $x\in\partial X_{\0}^N$ for a Morse gauge $N$. Then any geodesic ray $\alpha:[a,\infty)\rightarrow X$ with $\alpha(\infty)=x$ is $M$-Morse, where $M$ depends only on $N$ and the Morse gauge of $[\alpha(a),\0]$.
\end{prop}

\begin{proof}
    See Figure \ref{Figure2}. Let $\beta:[b,\infty)\rightarrow X$ be $N$-Morse with $\beta(b)=\0, \beta(\infty)=\alpha(\infty)$, and  let $N'$ be the Morse gauge of $[\alpha(a),\beta(b)]$. For each $n\in\mathbb{N}$, let $\gamma_n=[\alpha(a),\beta(b+n)]$. Note that $\beta|_{[b,b+n]}$ is Morse for some Morse gauge depending only on $N$ by Lemma \ref{MorseSubraysAreMorse}, and so by \cite[Lemma 2.3]{cordes_boundry2016}, $\gamma_n$ in $N''$-Morse for $N''$ depending only on $\max\{N,N'\}$. By potentially restricting to a subsequence for Arzel\`a-Ascoli and by \cite[Lemma 2.10]{cordes_boundry2016}, there exists an $N''$-Morse geodesic ray $\gamma$ with $\gamma_n\rightarrow \gamma$ (uniformly on compact sets) and $\gamma(\infty)=\beta(\infty)$. Then Proposition \ref{ParallelRaysWithSameStartareMorse} shows that $\alpha$ is Morse for an appropriate Morse gauge.
\end{proof}

\begin{figure}[h]
\centering
\begin{tikzpicture}
        \draw [->, black] (0,0) to (10,0);
        \draw node[label=below:{$\beta$}] at (6,0) {};
        \draw node[label=left:{$\beta(b)$},fill,circle,scale=0.5] at (0,0) {};
        \draw node[label=right:{$\alpha(a)$},fill,circle,scale=0.5] at (3,4) {};
        \draw [->, black] (3,4) to[out=-85,in=180] (10,0.5);
        \draw node[label=above:{$\alpha$}] at (6,0.9) {};
        \draw [-, black] (3,4) to[out=-130,in=40] (0,0);
        \draw [-, black] (3,4) to[out=-110,in=70] (2,0);
        \draw [-, black] (3,4) to[out=-90,in=120] (4,0);
        \draw node[label=below:{$\beta(b+1)$},fill,circle,scale=0.5] at (2,0) {};
        \draw node[label=below:{$\beta(b+2)$},fill,circle,scale=0.5] at (4,0) {};
        \draw [->, black] (3,4) to[out=-85,in=180] (10,0.1);
        \draw node[label=above:{$\gamma_2$}] at (4,0.15) {};
        \draw node[label=above:{$\gamma_1$}] at (2.5,0.15) {};
        \draw node[label=left:{$N'$-Morse}] at (2,2.5) {};
        \draw node[label=above:{$\gamma$}] at (7,0.15) {};
        \end{tikzpicture}\\
        \caption{Diagram for Proposition \ref{ParallelRaysAreMorse}}\label{Figure2}

\end{figure}



\subsection{Limit Sets and Weak Convex Hulls}

We now introduce limit sets and weak convex hulls, and give some useful properties that these sets have. We use these constructions to turn subsets of $X$ into subsets of the Morse boundary, and vice versa.

\begin{defn}\label{LimitSetDefinition}(\cite[Definition 3.2]{boundarycc2016}) Let $X$ be a proper, geodesic metric space and let $A\subseteq X$. The \textit{limit set of} $A$, denoted as $\Lambda A$, is the set of points in $\partial X_{\0}$ such that, for some Morse gauge $N$, there exists a sequence of points $(a_k)\subset A\cap X_{\0}^N$ such that $[\0,a_k]$ converges (uniformly on compact sets) to a geodesic ray $\alpha$ with $\alpha(\infty)=x$. (Note $\alpha$ is $N$-Morse by \cite[Lemma 2.10]{cordes_boundry2016}.) In the case where $H$ acts properly by isometries on $X$, we use $\Lambda H$ to denote the limit set of $H\0$.
\end{defn}

\begin{remark}
    By \cite[Lemma 3.3]{boundarycc2016}, $\Lambda H$ can be defined as the limit set of \textit{any} orbit of $H$, we merely choose the orbit $H\0$ for convenience and simplicity in future arguments. 
\end{remark}

We also prove a small fact about limit sets, which is similar to \cite[Lemma 4.1, Proposition 4.2]{boundarycc2016}.

\begin{cor}\label{LimitSetsInstratumAreCompact}
    Let $X$ be a proper, geodesic metric space and suppose $A\subseteq X$. If $\Lambda A\subseteq \partial X^N_{\0}$ for some Morse gauge $N$, then $\Lambda A$ is compact.
\end{cor}

\begin{proof}
    By \cite[Proposition 3.12]{cordes_boundry2016}, this is equivalent to the condition that $\Lambda H$ is closed.
\end{proof}

\begin{remark}
By Corollary \ref{LimitSetsInstratumAreCompact} and by \cite[Lemma 4.1]{boundarycc2016}, the requirement that $\Lambda H$ is compact is equivalent to the requirement that $\Lambda H$ is contained in the boundary of a single Morse stratum. 
\end{remark}

\begin{defn}\label{WeakConvexHullDefinition}(\cite{swenson2001, boundarycc2016})
    Let $X$ be a proper, geodesic metric space, and let $A\subseteq X\cup \partial X_{\0}$. Then the \textit{weak convex hull} of $A$, denoted $WCH(A)$, is the union of all geodesic (segments, rays, or lines) of $X$ which have both endpoints in $A$. 
\end{defn}

We take a moment to highlight some nice interactions between the weak convex hull of a compact limit set with the Morse boundary.

\begin{lemma}\label{LimitSetsInstratumHaveWCHInAstratum}(\cite[Proposition 4.2]{boundarycc2016}) Let $X$ be a proper geodesic metric space and let $A\subseteq X$ such that $\Lambda A\subseteq \partial X_{\0}^N$ for some Morse gauge $N$. Then there exists a Morse gauge $N'$, depending only on $N$, such that $WCH(\Lambda A)\subset X_{\0}^{N'}$. \hfill $\square$
\end{lemma}

\begin{lemma}\label{WeakConvexHullAndLimitSetContainments}
    Let $X$ be a proper geodesic metric space and let $A\subseteq X$ such that $\Lambda A\subseteq \partial X_{\0}^N$ for some Morse gauge $N$. Then $\Lambda(WCH(\Lambda A))\subseteq \Lambda A$.
\end{lemma}

\begin{proof}
    We may assume $|\Lambda A|>1$. Let $x\in \Lambda(WCH(\Lambda A))$. By Definition \ref{LimitSetDefinition}, there exists $x_n\in WCH(\Lambda A)$ such that $[\0,x_n]$ converges to a geodesic ray $\gamma$ with $\gamma(\infty)=x$. We show that there exists $K>0$ so that for all $n$ there exists $a_n$ with $[\0,x_n]\subseteq \Nhbd_K([\0,a_n])$. Thus, (a subsequence of) the geodesics $[\0,a_n]$ converge to a geodesic ray $\alpha:[0,\infty)\rightarrow X$ with $\alpha(0)=\0$, and $\alpha(\infty)=\gamma(\infty)=x$ and so $x\in\Lambda A$. It remains to find $K$ so that $[\0,x_n]\subseteq \Nhbd_K([\0,a_n])$.

    Fix $n$. Since $x_n\in WCH(\Lambda A)$, $x\in\eta$ where $\eta:(-\infty,\infty):\rightarrow X$ is a geodesic with $\eta(\pm\infty)\in\Lambda A$. So, by Definition \ref{LimitSetDefinition}, there exists $a_k^+,a_k^-\in A\cap X^N_{\0}$ so that $[\0,a_k^+]$ and $[\0,a_k^-]$ converge to geodesics $\beta^+$ and $\beta^-$, respectively, with $\beta^+(\infty)=\eta(\infty)$ and $\beta^-(-\infty)=\eta(-\infty)$. 
    Since $\Lambda A\subseteq \partial X_{\0}^N$, the triangle $\eta\cup\beta^+\cup\beta^-$ is $L$-slim for $L$ depending only on $N$ by \cite[Proposition 3.6]{boundarycc2016}, and as $x\in\eta$, there exists $y\in\beta^+\cup\beta^-$ so that $d(x_n,y)\leq L$. Without loss of generality, assume $y\in\beta^+$. Since $[\0,a_k^+]$ converges to $\beta^+$ uniformly on compact sets, choose $m$ large enough so that $d(y,[\0,a_m^+])\leq 1$. Let $z\in[\0,a_m^+]$ so that $d(y,z)\leq 1$. See Figure \ref{FigureLimitSetAndWCH}.

    Note that the concatenation $[\0,x_n]*[x_n,z]$ is a $(1,L+1)$-quasi-geodesic with endpoints on $[\0,a_m^+]$. Since $[\0,a_m^+]$ is $N$-Morse, we have that $[\0,x_n]\subseteq[\0,x_n]*[x_n,z]\subseteq \Nhbd_{N(1,L+1)}([\0,a_m^+])$. Since $K:=N(1,L+1)$ did not depend on the choice of $n$, this completes the proof.
\end{proof}

\begin{figure}[h]
\centering
\begin{tikzpicture}
        \draw[<->, black] (0,4) to (10,4);
        \draw node[label=above:{$\eta$}] at (2,4) {};
        \draw node[label=below:{$\beta^+$}] at (9,3.5) {};
        \draw node[label=below:{$\beta^-$}] at (0.5,3.5) {};
        \draw node[label=above:{$x_n$},fill,circle,scale=0.5] at (5,4) {};
        \draw node[label=above left:{$y$},fill,circle,scale=0.5] at (5,3.1) {};
        \draw [-] (5,4) to (5,3.1);
        \draw [-] (5,2.72) to (5,3.1);
        \draw node[label=below left:{$z$},fill,circle,scale=0.5] at (5,2.72) {};
        \draw [-, black] (2,0)  to[out=90,in=-150] (5,4);
        \draw [-, black] (2,0)  to[out=90,in=120] (7,2);
        \draw [->, black] (2,0) to[out=90,in=180] (10,3.5);
        \draw [->, black] (2,0) to[out=90,in=0] (0,3.5);
        \draw node[label=right:{$\leq L$}] at (5,3.5) {};
        \draw node[label=right:{$\leq 1$}] at (5,2.95) {};
        \draw node[label=right:{$a_m^+$},fill,circle,scale=0.5] at (7,2) {};
        \draw node[label=right:{$\0$},fill,circle,scale=0.5] at (2,0) {};
        \end{tikzpicture}\\
        \caption{Diagram for Lemma \ref{WeakConvexHullAndLimitSetContainments}}\label{FigureLimitSetAndWCH}

\end{figure}

Finally, we finish this section by stating the definitions of stability and boundary convex cocompactness here for reference.

\begin{defn} \label{stabledefinition}
    \textit{(\cite{Durham_Taylor_ConvexCC_and_Stab_2015} \cite[Definition 1.3]{boundarycc2016})} If $f:X\rightarrow Y$ is a quasi-isometric embedding between geodesic metric spaces, we say $X$ is a \textit{stable} subspace of $Y$ if there exists a Morse Gauge $N$ such that every pair of points in $X$ can be connected by an $N$-Morse quasigeodesic in $Y$; we call $f$ a \textit{stable embedding}.
    
    If $H< G$ are finitely generated groups, we say $H$ is \textit{stable in} $G$ if the inclusion map $i:H\hookrightarrow G$ is a stable embedding.
\end{defn}

\begin{defn}\label{BoundaryConvexCocompactDefinition} \textit{(\cite[Definition 1.4]{boundarycc2016})}
    We say that $H$ acts \textit{boundary convex cocompactly} on $X$ if the following conditions hold:
    \begin{enumerate}
        \item $H$ acts properly on $X$,
        \item $\Lambda H$ is nonempty and compact,
        \item The action of $H$ on $WCH(\Lambda H)$ is cobounded.
    \end{enumerate}
\end{defn}


\section{Limit point characterizations in the Morse Boundary}

The goal of this section is show that, given a set $A\subseteq X$, if $x\in\partial X_{\0}$ is a Morse conical limit point of $A$, then $x$ is a Morse horospherical limit point of $A$. This was first shown in the hyperbolic case in \cite{swenson2001}, here we generalize this fact into the setting of proper geodesic metric spaces. We begin by introducing definitions which generalize horospheres and funnels for Morse rays.

\begin{defn}[Horoballs, Funnels]
    Let $X$ be a proper, geodesic metric space and let $\0\in X$ be some designated point. Let $\alpha:[a,\infty)\rightarrow X$ be an $N'$-Morse geodesic ray, and let $N$ be some, potentially different, Morse gauge. We define the \textit{$N$-Morse horoball around $\alpha$ based at $\0$} as 
    \[H^N_{\0}(\alpha)=\{x\in X^N_{\0}~|~\exists \beta:[b,\infty)\rightarrow X \text{ with }\beta\sim_{\delta_{N'}}\alpha\text{ and }b\geq a\text{ and }\beta(b)=x\}. \]
    We define the \textit{$N$-Morse funnel around $\alpha$ based at $\0$} as
    \[F^N_{\0}(\alpha)=\{x\in X^N_{\0}~|~d(x,\pi_\alpha(x))\leq d(\alpha(a),\pi_\alpha(x))\}.\]
\end{defn}

Comparing these definitions to Definition \ref{hyperbolicdef1} shows that a Morse horoball is a horoball about a Morse geodesic intersected with an appropriate Morse stratum, and similarly, a Morse funnel is a funnel about a Morse geodesic intersected with an appropriate Morse stratum. The following three definitions classify points on the Morse boundary by asking if every horoball, funnel, or cone intersects a given subset of $X$.

\begin{defn}\label{conical limit point def}
    Let $X$ be a proper, geodesic metric space and let $\0\in X$ be some designated point. Let $A\subset X$.
    
    \begin{itemize}
    \item We say that $x\in\partial X_{\0}$ is a \textit{Morse horospherical limit point of $A$} if for every Morse geodesic $\alpha$ with $\alpha(\infty)=x$, there exists a Morse gauge $N$ such that $H^{N}_{\0}(\alpha)\cap A\neq\emptyset$.
    
    \item We say that $x\in\partial X_{\0}$ is a \textit{Morse funneled limit point of $A$} if for every Morse geodesic $\alpha$ with $\alpha(\infty)=x$, there exists a Morse gauge $N$ such that $F^{N}_{\0}(\alpha)\cap A\neq\emptyset$.
    
    \item We say that $x\in\partial X_{\0}$ is a \textit{Morse conical limit point of $A$} if there exists $K>0$ such that, for every Morse geodesic $\alpha$ with $\alpha(\infty)=x$, we have that $\Nhbd_K(\alpha)\cap A\neq\emptyset$.
    \end{itemize}
\end{defn}

\begin{remark} Notice that, in the case where $X$ is a $\delta$-hyperbolic space, these definitions agree with the definitions given in Definition \ref{hyperbolicdef2}, as every geodesic in a $\delta$-hyperbolic space is $N$-Morse for $N$ depending only on $\delta$. In light of this, we will use ``conical limit point" instead of ``Morse conical limit point" for the rest of this paper, except in cases where the difference between these definitions causes confusion. We similarly reduce ``Morse horospherical limit point" and ``Morse funneled limit point" to ``horospherical limit point" and ``funneled limit point," respectively. 
\end{remark}

We now begin proving the new implications found in Theorem \ref{MyMainTheorem}. We will first show that every conical limit point of $A$ is a funneled limit point of $A$, and then we will show that the funneled limit points of $A$ exactly coincide with the horospherical limit points of $A$. These arguments generalize the arguments found in \cite{swenson2001}.

\begin{prop}\label{MorseConicalImpliesMorseFunnelled}
    Let $X$ be a proper, geodesic metric space and let $\0\in X$. Let $A\subseteq X$. If $x\in\partial X_{\0}$ is a conical limit point of $A$, then $x$ is a funneled limit point of $A$.
\end{prop}

\begin{proof}
    See Figure \ref{Figure4}. Let $x\in\partial X_{\0}$ be a conical limit point of $A\subseteq X$. Let $\alpha:[0,\infty)\rightarrow X$ be an $N$-Morse geodesic with $\alpha(\infty)=x$. By Lemma \ref{MorseSubraysAreMorse}, there exists a Morse gauge $M$ so that every geodesic sub-ray of $\alpha$ is $M$-Morse. Thus by Definition \ref{conical limit point def}, there exists $K\geq 0$ so that every subray of $\alpha$ gets at least $K$ close to $A$. 
    
    Now define $\alpha'=\alpha|_{[3K,\infty)}$,  and let $a\in A$ such that $a\in \Nhbd_K(\alpha')$. Then note that $d(a,\pi_\alpha(a))\leq d(a,\pi_{\alpha'}(a))\leq K$, and so $d(\pi_\alpha(a),\pi_{\alpha'}(a))\leq 2K$. By the triangle inequality, $\pi_\alpha(a)\subseteq \alpha|_{[K,\infty)}$. Therefore, $d(\pi_{\alpha}(a),a)\leq K=d(\alpha(0),\alpha(K)) \leq d(\alpha(0),\pi_{\alpha'}(a))$. It remains to show that $a\in X_{\0}^{N'}$ for a Morse gauge $N'$ which is independent of the choice of $a\in A$.
    
    Let $L=d(\0,\alpha(0))$, and let $p\in\pi_{\alpha'}(a)$, and note that $d(p,a)\leq K$. Thus, $[\0,\alpha(0)]$ and $[p,a]$ are both $N''$-Morse depending only on $\max\{K,L\}$, and $[\0,p]$ is $N'''$-Morse depending only on $N$ by Lemma \ref{MorseSubraysAreMorse}. Since $[\0,a]$ is one side of a quadrilateral whose other three sides are $\max\{N'',N'''\}$-Morse, $[\0,a]$ is $N'$-Morse where $N'$ does not depend on choice of $a\in A$ by \cite[Lemma 2.3]{cordes_boundry2016}.
\end{proof}

\begin{figure}[h]
\centering
\begin{tikzpicture}
        \draw [->, black] (0,0) to (10,0);
        \draw node[label=right:{$\alpha(\infty)=x$}] at (10,0) {};
        \draw node[label=left:{$\alpha(0)$},fill,circle,scale=0.5] at (0,0) {};
        \draw node[label=below:{$\alpha(3K)$},fill,circle,scale=0.5] at (6,0) {};
        \draw node[label=above :{\textcolor{blue}{$\pi_{\alpha'}(a)$}}] at (6.5,-1.2) {};
        \draw [-, blue, ultra thick] (6,0) to (7,0);
        \draw node[label=above :{\textcolor{blue}{$\pi_{\alpha}(a)$}}] at (5,-1.2) {};
        \draw [-, blue, ultra thick] (4.5,0) to (5.5,0);
        \draw[-,dashed,black] (4,0) to[out=90,in=180] (6,2);
        \draw[-,dashed,black] (4,0) to[out=-90,in=180] (6,-2);
        \draw[-,dashed,black] (6,2) to (9,2);
        \draw[-,dashed,black] (6,-2) to (9,-2);
        \draw node[label=below:{$\Nhbd_K(\alpha')$}] at (8,2) {};
        \draw node[label=above:{$a$},fill,circle,scale=0.5] at (5,1) {};
        \draw[-,dashed,black] (5,1) to (4.5,0);
        \draw[-,dashed,black] (5,1) to (5.5,0);
        \draw[-,dashed,black] (5,1) to (6,0);
        \draw[-,dashed,black] (5,1) to (7,0);
        \draw node[label=right:{$\leq K$}] at (6,0.7) {};
        \draw node[label=above:{$\0$},fill,circle,scale=0.5] at (2,1) {};
        \draw[-] (0,0) to (2,1);
        \draw[-] (2,1) to (5,1);
        \draw node[label=above:{$L$}] at (1,0.5) {};
        \draw node[label=above right:{$p$},fill,circle,scale=0.5] at (7,0) {};
        \end{tikzpicture}\\
        \caption{Diagram for Proposition \ref{MorseConicalImpliesMorseFunnelled}}\label{Figure4}

\end{figure}

Our next goal is to show that the funneled limit points of $A$ coincide with the horospherical limit points of $A$. Towards this end, we show that, given a point $x$ in a horoball of a subray, the projection of $x$ to the subray is coarsely the same as the projection to the base ray.

\begin{lemma}\label{HoroballProjectionToSubrayCloseToParentProjection}
Suppose $\alpha$ is an $N$-Morse geodesic ray and let $\alpha'$ be a subray. Suppose $x\in H_{\0}^{N'}(\alpha')$. If $\alpha(\infty)\in\partial X_{\0}^{N''}$, then  $d_{Haus}(\pi_\alpha(x),\pi_{\alpha'}(x))\leq K$, where $K\geq0$ depends only on $N$, $N'$, and $N''$.
\end{lemma}

\begin{proof}
    See Figure \ref{Figure5}. Let $\alpha:[0,\infty)\rightarrow X$ be an $N$-Morse geodesic ray and let $\alpha'=\alpha|_{[a,\infty)}$ for some $a\geq0$. By Lemma \ref{MorseSubraysAreMorse}, $\alpha'$ is $M$-Morse for $M$ depending only on $N$. Let $x\in H_{\0}^{N''}(\alpha')$, thus there exists $\beta:[b,\infty)\rightarrow X$ a geodesic ray with $b\geq a$, $\beta(b)=x$, and $\beta\sim_{\delta_M}\alpha'$. By Proposition \ref{ParallelRaysAreMorse}, $\beta$ is $M'$-Morse for $M'$ depending only on $N$, $N'$, and $N''$. We note that if $\pi_\alpha(x)\subseteq \alpha'$, then $\pi_\alpha(x)=\pi_{\alpha'}(x)$. 
    So, we assume that $\pi_\alpha(x)\not\subseteq\alpha'$. We shall show that in this case, $d(x,\pi_\alpha(x))$ and $d(x,\pi_{\alpha'}(x))$ are both bounded above by an appropriate constant, and this gives the desired result.
    
    Let $p\in\pi_\alpha(x)\setminus\alpha'$, and let $q\in\pi_{\alpha'}(x)$. Without loss of generality, let  $T$ be large enough so that $q\in[\alpha'(a),\alpha'(T)]$ and $d(\alpha'(T),\beta(T))\leq \delta_N$.
    
    Put $\gamma=[\beta(b),\alpha'(a)]*[\alpha'(a),\alpha'(T)]*[\alpha'(T),\beta(T)]$, and note that $\gamma$ is a $(3,4\delta_N)$ quasi-geodesic. Thus there exists $w\in[\beta(b),\beta(T)]$ and $L\geq0$ such that $d(\alpha'(a),w)\leq L$, where $L$ depends only on $M'$ by \cite[Lemma 2.1]{cordes_boundry2016}. Notice now that $|d(\alpha'(a),\alpha'(T))-d(w,\beta(T))|\leq\delta_N+L$. However, since $b\geq a$ and $w\in[\beta(b),\beta(T)]$, we know $|d(\alpha'(a),\alpha'(T))-d(w,\beta(T))|=d(\alpha'(a),\alpha'(T))-d(w,\beta(T))$. But then by the definition of the nearest point projection and the triangle inequality, we have 
    \begin{align*}
    d(x,p)&\leq d(x,q)\leq d(x,\alpha'(a))\leq d(x,w)+d(w,\alpha'(a)) =d(x,\beta(T))-d(w,\beta(T))+d(w,\alpha'(a))\\ 
    &=d(\alpha'(b),\alpha'(T))-d(w,\beta(T))+d(w,\alpha'(a))\leq d(\alpha'(a),\alpha'(T))-d(w,\beta(T))+L\leq \delta_N+L+L.
    \end{align*}
    Therefore, $d(\pi_\alpha(x),x)$ and $d(\pi_{\alpha'}(x),x)$ are both bounded above by $L$, which is a constant depending only on $N$, $N'$, and $N''$, as desired.
\end{proof}

\begin{figure}[h]
\centering
\begin{tikzpicture}
        \draw [->, black] (0,0) to (10,0);
        \draw node[label=below:{$\alpha$}] at (1,0) {};
        \draw node[label=left:{$\alpha(0)$},fill,circle,scale=0.5] at (0,0) {};
        \draw[-,dashed,black] (2,0) to[out=90,in=180] (6,4);
        \draw[-,dashed,black] (2,0) to[out=-90,in=180] (6,-4);
        \draw node[label=above left:{$H^{N'}_{\0}(\alpha')$}] at (3,2.6) {};
        \draw node[label=above :{\textcolor{red}{$\pi_{\alpha'}(x)$}}] at (6.5,-1.5) {};
        \draw [-, red, ultra thick] (6,0) to (7,0);
        \draw node[label=above :{\textcolor{blue}{$\pi_{\alpha}(x)$}}] at (3.5,-1.5) {};
        \draw [-, blue, ultra thick] (3,0) to (4,0);
        \draw node[label=below:{$\alpha(a)$},fill,circle,scale=0.5] at (5,0) {};
        \draw node[label=below:{$=\alpha'(a)$}] at (5,-0.4) {};
        \draw[-, black] (5,3) to[out=-80,in=180] (8,1);
        \draw[->,black] (8,1) to (9.9,1);
        \draw node[label=above:{$\beta(T)$},circle,fill,scale=0.5] at (8,1) {};
        \draw node[label=below:{$\alpha(T)$},circle,fill,scale=0.5] at (8,0) {};
        \draw node[label=below:{$=\alpha'(T)$}] at (8,-0.4) {};
        \draw[-,black] (8,1) to (8,0);
        \draw node[label=right:{$\leq \delta_N$}] at (8,0.5) {};
        \draw node[label=above:{$w$},fill,circle,scale=0.5] at (6,1.52) {};
        \draw [-,red] (5,3) to[out=-90, in=170] (6.5,0);
        \draw node[label=below:{$q$},circle,fill,scale=0.5] at (6.5,0) {};
        \draw [-,blue] (5,3) to (3,0);
        \draw node[label=right:{$\beta(b)=x$},fill,circle,scale=0.5] at (5,3) {};
        \draw node[label=below:{$p$},circle,fill,scale=0.5] at (3,0) {};
        \draw [-,black] (5,0) to (6,1.52);
        \draw node[label=right:{$\leq L$}] at (5.6,1) {};
        \end{tikzpicture}\\
        \caption{Diagram for Lemma \ref{HoroballProjectionToSubrayCloseToParentProjection}}
        \label{Figure5}

\end{figure}

We're now ready to show that Morse funneled limit points are exactly Morse horospherical limit points. We proceed using the same overall strategy as the one found in \cite{swenson2001}, by showing direct generalizations of Lemma \ref{funnel_in_horoball_hyperbolic} and Lemma \ref{horoball_in_funnel_hyperbolic} for the Morse case.

\begin{prop}\label{funnelscontainhoroballs}
    Let $x\in\partial X_{\0}^N$. Let $\alpha:[0,\infty)\rightarrow X$ be an $N'$-Morse geodesic with $\alpha(\infty)=x$. Then for every Morse gauge $N''$, there exists $T\geq0$ such that, for any subray $\alpha'$ of $\alpha$ with $d(\alpha(0),\alpha')\geq T$, we have $H_{\0}^{N''}(\alpha')\subseteq F_{\0}^{N''}(\alpha)$.
\end{prop}

\begin{proof}
    See Figure \ref{Figure6}. Let $\alpha'=\alpha|_{[a,\infty)}$ be a subray of $\alpha$. By Lemma \ref{MorseSubraysAreMorse}, $\alpha'$ is $M$-Morse where $M$ depends only on $N'$. Let $y\in H^{N''}_{\0}(\alpha')$. Thus there exists $\beta:[b,\infty)\rightarrow X$ be a geodesic ray such that $b\geq a$, $\beta(b)=y$, and $\beta\sim_{\delta_{M}}\alpha'$. Note that $\beta$ is $M'$-Morse where $M'$ depends only on $N$, $N'$, and $N''$ by Proposition \ref{ParallelRaysAreMorse}. Choose $z\in\pi_\alpha(y)$ such that $d(\alpha(0),z)=d(\alpha(0),\pi_\alpha(y))$, i.e., so that $z$ is closest to $\alpha(0)$. 
    By Lemma \ref{HoroballProjectionToSubrayCloseToParentProjection}, there exists $p\in\pi_{\alpha'}(x)$ so that $d(z,p)\leq L$ for some $L$ depending only on $N$, $N'$, and $N''$. Choose $t$ large enough so that $d(\alpha'(t),\beta(t))=d(\alpha(t),\beta(t))\leq\delta_M$ and $p,\alpha(b)\in[\alpha(a),\alpha(t)]$. Note that $[y,p]*[p,\alpha(t)]*[\alpha(t),\beta(t)]$ is a $(3,4\delta_M)$-quasi-geodesic, thus there exists $q\in[\beta(b),\beta(t)]$ and $\lambda\geq0$ such that $d(p,q)\leq\lambda$ where $\lambda$ depends only on $M'$ by \cite[Lemma 2.1]{cordes_boundry2016}. It suffices to show that $d(y,z)\leq d(\alpha(0),z)$.
    
    Using the triangle inequality and the definition of $\pi_{\alpha}$, we find
    \begin{align*}
    d(y,z) &\leq d(y,p)\leq d(y,q)+d(q,p)\leq d(y,q)+\lambda\\ 
    &= d(y,\beta(t))-d(q,\beta(t))+\lambda=d(\alpha(b),\alpha(t))-d(q,\beta(t))+\lambda\\
    &\leq d(\alpha(a),\alpha(t))-d(p,\alpha(t))+\lambda+\delta_M+\lambda
    =d(\alpha(a),p)+2\lambda+\delta_M\\
    &\leq d(\alpha(a),z)+L+2\lambda+\delta_M.
    \end{align*}
    
    So, if $a\geq L+2\lambda+\delta_M$, we have \[d(y,z)\leq d(\alpha(a),z)+L+2\lambda+\delta_M\leq d(\alpha(a),z)+d(\alpha(0),\alpha(a))=d(\alpha(0),z).\qedhere\]
\end{proof}

\begin{figure}[h]
\centering
\begin{tikzpicture}
        \draw [->, black] (0,0) to (10,0);
        \draw node[label=below:{$\alpha$}] at (1,0) {};
        \draw node[label=left:{$\alpha(0)$},fill,circle,scale=0.5] at (0,0) {};
        \draw [-,dashed] (0,0) to (4,4);
        \draw [-,dashed] (0,0) to (4,-4);
        \draw node[label=above left:{$F^{N'}_{\0}(\alpha)$}] at (2,2) {};
        \draw[-,dashed,black] (2,0) to[out=90,in=180] (6,4);
        \draw[-,dashed,black] (2,0) to[out=-90,in=180] (6,-4);
        \draw node[label=right:{$H^{N'}_{\0}(\alpha')$}] at (6,3.5) {};
        \draw node[label=above :{\textcolor{red}{$\pi_{\alpha'}(y)$}}] at (6.5,-1.5) {};
        \draw [-, red, ultra thick] (6,0) to (7,0);
        \draw node[label=above :{\textcolor{blue}{$\pi_{\alpha}(y)$}}] at (3.5,-1.5) {};
        \draw [-, blue, ultra thick] (3,0) to (4,0);
        \draw node[label=below:{$\alpha(a)$},fill,circle,scale=0.5] at (5,0) {};
        \draw node[label=below:{$=\alpha'(a)$}] at (5,-0.4) {};
        \draw[-, black] (5,3) to[out=-80,in=180] (8,1);
        \draw[->,black] (8,1) to (9.9,1);
        \draw node[label=above:{$\beta(T)$},circle,fill,scale=0.5] at (8,1) {};
        \draw node[label=below:{$\alpha(T)$},circle,fill,scale=0.5] at (8,0) {};
        \draw node[label=below:{$=\alpha'(T)$}] at (8,-0.4) {};
        \draw[-,black] (8,1) to (8,0);
        \draw node[label=right:{$\leq \delta_M$}] at (8,0.5) {};
        \draw node[label=above:{$q$},fill,circle,scale=0.5] at (6,1.52) {};
        \draw [-,red] (5,3) to[out=-90, in=170] (6.5,0);
        \draw node[label=below:{$p$},circle,fill,scale=0.5] at (6.5,0) {};
        \draw [-,blue] (5,3) to (3,0);
        \draw node[label=right:{$\beta(b)=y$},fill,circle,scale=0.5] at (5,3) {};
        \draw node[label=below:{$z$},circle,fill,scale=0.5] at (3,0) {};
        \draw [-,black] (6.5,0) to (6,1.52);
        \draw node[label=right:{$\leq \lambda$}] at (6.2,0.8) {};
        \end{tikzpicture}\\
        \caption{Diagram for Proposition \ref{funnelscontainhoroballs}}\label{Figure6}
\end{figure}

\begin{prop}\label{horoballscontainfunnels}
    Let $x\in\partial X^N_{\0}$. Let $\alpha:[0,\infty)\rightarrow X$ be an $N'$-Morse geodesic with $\alpha(\infty)=x$. Suppose $S=\delta_{N'}$. Define $\alpha'=\alpha|_{[S,\infty)}$. Then $F^{N''}_{\0}(\alpha')\subseteq H^{N''}_{\0}(\alpha)$ for any Morse gauge $N''$.
\end{prop}

\begin{proof}
    Let $y\in F_{\0}^{N''}(\alpha')$. By definition, $d(y,\pi_{\alpha'}(y))\leq d(\alpha(S),\pi_{\alpha'}(y))$. 
    Let $p\in\pi_{\alpha'}(y)$ such that $d(\alpha(S),p)=d(\alpha(S),\pi_{\alpha'}(y))$, i.e., let $p$ be the element of $\pi_{\alpha'}(y)$ which is closest to $\alpha(S)$. Then $d(y,p)\leq d(\alpha(S),p)$. Construct $\beta:[b,\infty)\rightarrow X$ such that $\beta(b)=y$ and $\beta\sim_{\delta_{N'}}\alpha$. We want to show that $b\geq0$. Choose $T\geq0$ so that $d(\beta(T),\alpha(T))\leq\delta_{N'}$. Then
    \begin{align*}
    T-b &=d(y,\beta(T))\leq d(y,p)+d(p,\alpha(T))+d(\alpha(T),\beta(T))\\
    &\leq d(\alpha(S),p)+d(p,\alpha(T))+\delta_{N'}
    = d(\alpha(S),\alpha(T))+\delta_{N'}\\
    &= d(\alpha(0),\alpha(T))-d(\alpha(0),\alpha(S))+\delta_{N'} =T-S+\delta_{N'}=T.
    \end{align*}
    In summary, $T-b\leq T$, but this immediately shows that $0\leq b$, as desired.
\end{proof}

\begin{figure}[h]
\centering
\begin{tikzpicture}
        \draw [->, black] (0,0) to (10,0);
        \draw node[label=below:{$\alpha$}] at (7,0) {};
        \draw node[label=left:{$\alpha(0)$},fill,circle,scale=0.5] at (0,0) {};
        \draw node[label=right:{$y=\beta(b)$},fill,circle,scale=0.5] at (5,3) {};
        \draw node[label=below:{$\alpha(S)$},fill,circle,scale=0.5] at (1,0) {};
        \draw node[label=below:{$=\alpha'(S)$},scale=0.5] at (1,-0.5) {};
        \draw node[label=below:{$\alpha(T)$},fill,circle,scale=0.5] at (9,0) {};
        \draw node[label=above:{$\beta(T)$},fill,circle,scale=0.5] at (9,0.5) {};
        \draw [->, black]   (5,3) to[out=-85,in=180] (10,0.5);
        \draw [-, black,dashed]   (-1,0) to[out=90,in=180] (5,5);
        \draw [-, black,dashed]   (-1,0) to[out=-90,in=117] (0,-3);
        \draw node[label=above:{$\beta$}] at (7,0.9) {};
        \draw [-, black] (5,3) to (4,0);
        \draw [-, black] (5,3) to (5,0);
        \draw [-, blue, ultra thick] (4,0) to (5,0);
        \draw node[label=below:{\textcolor{blue}{$\pi_{\alpha'}(y)$}}, text=blue] at (5,0) {};
        \draw node[label=below:{$p$},fill,circle,scale=0.5] at (4,0) {};
        \draw [-, black, dashed] (6,5) to (1,0);
        \draw node[label=right:{$F^{N''}_{\0}(\alpha')$},scale=0.5] at (6,4.5) {};
        \draw node[label=right:{$H^{N''}_{\0}(\alpha)$},scale=0.5] at (3,4.5) {};
        \draw [-, black, dashed] (4,-3) to (1,0);
        \draw [-, black] (9,0.5) to (9,0);
        \draw node[label=right:{$\leq\delta_{N'}$}] at (9,0.25) {};
        \end{tikzpicture}\\
        \caption{Diagram for Proposition \ref{horoballscontainfunnels}}\label{Figure7}
\end{figure}

\begin{thm}\label{MorseHoroLimitsAreMorseFunnelLimits}
    Let $x\in\partial X_{\0}$. Then $x$ is a Morse horospherical limit point of $A\subseteq X$ if and only if $x$ is a Morse funneled limit point of $A$. 
\end{thm}

\begin{proof}
    Let $x\in\partial X_\0$. Then there exists a Morse gauge $N$ so that $x\in\partial X^N_\0$. Let $\alpha$ be any Morse geodesic with $\alpha(\infty)=x$, and let $H^{N''}_\0(\alpha)$ and $F^{N''}_\0(\alpha)$ be a Morse horoball about $\alpha$ and a Morse funnel about $\alpha$, respectively. By Propositions \ref{funnelscontainhoroballs} and \ref{horoballscontainfunnels}, there exists a subray $\alpha'$ so that $F^{N''}_\0(\alpha')\subseteq H^{N''}_\0(\alpha)$ and $H^{N''}_\0(\alpha')\subseteq F^{N''}_\0(\alpha)$.

    Now suppose $x$ is horospherical. Then there exists $a\in A$ so that $x\in H^{N''}_\0(\alpha')\subseteq F^{N''}_\0(\alpha)$, and as the funnel $F^{N''}_\0(\alpha)$ was arbitrary, $x$ is funneled. Similarly, suppose $x$ is funneled. Then there exists $a\in A$ so that $x\in F^{N''}_\0(\alpha')\subseteq H^{N''}_\0(\alpha)$, and as the funnel $F^{N''}_\0(\alpha)$ was arbitrary, $x$ is horospherical.
\end{proof}


\section{Limit Set Conditions For Stability}

In this section, we show that the horospherical limit point condition, combined with the limit set being compact, is enough for to show that the group action on the weak convex hull is cobounded. The main idea behind this argument is to show the contrapositive: when the group action is not cobounded, then geodesic rays in the space eventually end up very far from the orbit of the group. We begin by showing the following helpful fact, which states that if a group acts non-coboundedly on the weak convex hull of its limit set, there exists a sequence of points $p_n$ in the weak convex hull that ``maximally avoids" the orbit.

\begin{lemma}\label{noncoboundedactionsarewellspreadinweakconvexhull}
    Suppose $X$ is a proper geodesic metric space and suppose that $H$ acts properly on $X$ by isometries. Assume that $\Lambda H\neq\emptyset$. If the action $H\curvearrowright WCH(\Lambda H)$ is not cobounded, then there exists an increasing sequence of positive integers, $(n_i)_{i}$, such that for each $i\in\mathbb{Z}_{\geq1}$ there exists $p_i\in WCH(\Lambda H)$ satisfying
    \begin{enumerate}
        \item $B_{n_i}(p_i)\cap H\0=\emptyset$,
        \item $d(p_i,\0)\leq n_i+1$.
    \end{enumerate}
\end{lemma}

\begin{proof}
    Set $n_0=1$. We define $q_i$ and $n_i$ for $i\geq 1$ via an inductive process. Since the action of $H\curvearrowright WCH(\Lambda H)$ is not cobounded, there exists a point $q_i\in WCH(\Lambda H)$ such that $n_{i-1}+1<d(H\0,q_i)$. By the definition of $WCH(\Lambda)$, there exists a bi-infinite Morse geodesic $\gamma$ with $\gamma(\pm\infty)\in\Lambda H$ such that $q_i\in\gamma$. Set $n_i$ to be the unique positive integer such that $n_i<d(H\0,q_i)\leq n_i+1$. Note that the sequence $(n_i)_i$ is increasing because $n_{i-1}+1\leq n_i$. 
    
    Since $d(H\0,q_i)\leq n_i+1$ there exists $h_i\in H$ so that $d(q_i,h_i\0)\leq n_i+1$. Recalling that the action of $H$ on $X$ is by isometries, we define $p_i=h_i^{-1}q_i$, and so $B_{n_i}(p_i)\cap H\0=\emptyset$,
	and $d(\0,p_i)\leq n_i+1$. Finally, by \cite[Lemma 3.3]{boundarycc2016}, $h_i^{-1}\gamma$ is a bi-infinite Morse geodesic with endpoints in $\Lambda H$, and so $p_i\in WCH(\Lambda H)$.
\end{proof}

We note that, under the additional assumption that $\Lambda H$ is compact and that every point in $\Lambda H$ is conical, we get a stronger conclusion to this lemma, namely, we can take $n_i=i$ for large $i$. We formally state and prove this observation.

\begin{lemma}[Sliding Spheres]\label{sliding spheres}
    Suppose $X$ is a proper geodesic metric space, and suppose that $H$ acts properly on $X$ by isometries.
    Assume that $\Lambda H\neq\emptyset$, every point of $\Lambda H$ is a conical limit point of $H\0$, and that $\Lambda H\subseteq \partial X_{\0}^N$ for some Morse gauge $N$.
    If the action $H\curvearrowright WCH(\Lambda H)$ is not cobounded, 
    there exists a sequence of points $p_n\in WCH(\Lambda H)$ such that, for sufficiently large $n$, $B_n(p_n)\cap H\0=\emptyset$ and $\0\in B_{n+1}(p_n)$.
\end{lemma}

\begin{proof}
    Let $K>0$ be the conical limit point constant. Let $n\in\mathbb{N}$ with $n>K+1$. By \cite[Corollary 5.8]{liu_2021}, we may assume that $\Lambda H$ has at least two distinct points. 
    Since $H\curvearrowright WCH(\Lambda H)$ is not cobounded, there exists $p\in WCH(\Lambda H)$ 
	with $d(p,H\0)>n$. By definition, $p\in\gamma$ for some bi-infinite geodesic $\gamma$ with 
	$\gamma(\pm \infty)\in\Lambda H$. Since $\Lambda H\subseteq \partial X_{\0}^N$, we have by \cite[Proposition 4.2]{boundarycc2016} that $\gamma$ is is Morse for some Morse gauge depending only on $N$. Since every point in $\Lambda H$ is a conical limit point of $H\0$,
	there exists $h'\in H$ such that $d(h'\0,\gamma)<K$. Put $q\in\pi_\gamma(h'\0)$.
	 
	We may assume that $\gamma(s)=q$ and $\gamma(s')=p$ with $s<s'$.
	Let $A=\{r\in[s,s']~:~n<d(\gamma(r),H\0)\}.$ (Equivalently, one may define
	$A=\{r\in[s,s']~:~B_n(\gamma(r))\cap H\0=\emptyset\}$.) Note that $s'\in A$.
	Put $t=\inf A$. By the definition of $t$, we have $n\leq d(\gamma(t),H\0)$. See Figure \ref{FigureBallsAvoidingNonCoboundedActions}. We now claim that $d(\gamma(t),H\0)<n+1$. 
	 
	Suppose for contradiction that $n+1\leq d(\gamma(t),H\0)$. By the triangle inequality,
	$n\leq d(\gamma(t-1),H\0)$. So if $t-1\in[s,s']$, then $t-1\in A$, however $t=\inf A$. Thus $t-1\not\in [s,s']$.
	Therefore, $t\in[s,s+1]$, and so 
    \[n+1\leq d(\gamma(t),h\0)\leq d(\gamma(t),h'\0)\leq d(\gamma(t),q)+d(q,h'\0)
	= d(\gamma(t),\gamma(s))+d(q,h'\0)\leq 1+K\leq n,\] a contradiction.
	 
	Therefore, there exists $h\in H$ such that $h\0\in B_{n+1}(\gamma(t))$, but $B_n(\gamma(t))\cap H\0=\emptyset$.
	Put $p_n=h^{-1}(\gamma(t))$. By \cite[Lemma 3.3]{boundarycc2016}, $h\gamma$ is a bi-infinite Morse geodesic with endpoints in $\Lambda H$, and since the action of $H$ on $X$ is by isometries, $B_n(p_n)\cap H\0=\emptyset$,
	and $\0\in B_{n+1}(p_n)$, as desired.
\end{proof}

\begin{figure}[h]
\centering
\begin{tikzpicture}
        \draw [<->, black] (0,0) to (10,0);
        \draw [-] (8,0) to (8,1.5);
        \draw [-] (4,0) to (4,1.5);
        \draw node[label=right:{$n$}] at (4,0.75) {};
        \draw node[label=right:{$n$}] at (8,0.75) {};
        \draw node[label=below:{\textcolor{red}{$\gamma(A)$}}] at (6,0) {};
        \draw [-, red, ultra thick] (4,0) to (8,0);
        \draw node[label=below:{$\gamma(s')=p$},fill,circle,scale=0.5] at (8,0) {};
        \draw node[label=below:{$\gamma(t)$},fill,circle,scale=0.5] at (4,0) {};
        \draw [dashed] (8,0) circle [radius=1.5];
        \draw [dashed] (4,0) circle [radius=1.5];
        \draw node[label=above:{$h\0$},fill,circle,scale=0.5] at (2.6,1) {};
        \draw node[label=above:{$\gamma(s)=q$},fill,circle,scale=0.5] at (1,0) {};
        \draw node[label=below:{$h'\0$},fill,circle,scale=0.5] at (1,-0.5) {};
        \draw[-] (1,0) to (1,-0.5);
        \draw node[label=right:{$H\0$},fill,circle,scale=0.5] at (0.5,2) {};
        \end{tikzpicture}\\
        \caption[Diagram for Lemma \ref{sliding spheres}.] {Diagram for Lemma \ref{sliding spheres}. We can think of this proof as sliding the ball on the right towards the left until it is ``up against" the orbit $H\0$, such as the ball centered at $\gamma(t)$.}\label{FigureBallsAvoidingNonCoboundedActions}
\end{figure}

We now prove that (4) implies (2) in the language of Theorem \ref{MyMainTheorem}. We show that, if the action is not cobounded on the weak convex hull, then using Lemma \ref{noncoboundedactionsarewellspreadinweakconvexhull} we can find a sequence of points $p_i$ which maximally avoid the orbit of $H$. However, this sequence of points defines a new ray $\gamma$ with $\gamma(\infty)\in\Lambda H$. Then using the horospherical point assumption, we find an orbit point close to $p_i$, a contradiction.

\begin{thm} \label{AllMorseHoroImpliesCoboundedAction}

    Suppose $X$ is a proper geodesic metric space and suppose $H$ acts properly on $X$ by isometries.
	Assume that $\Lambda H\neq \emptyset$, every point of $\Lambda H$ is a horospherical limit point of $H\0$, and that there exists a Morse gauge $N$ such that $\Lambda H\subset \partial X^N_{\0}$.
	Then the action of $H\curvearrowright WCH(\Lambda H)$ is cobounded.

\end{thm}

\begin{proof}
	For contradiction, assume that $H\curvearrowright WCH(\Lambda H)$ is not a cobounded action. 
	By Lemma \ref{noncoboundedactionsarewellspreadinweakconvexhull}, there exists a sequence of points $p_i\in WCH(\Lambda H)$ and an increasing sequence of positive integers $(n_i)_i$ such that $B_{n_i}(p_i)\cap H\0=\emptyset$, and $\0\in B_{n_i+1}(p_i)$. 
	Let $\gamma_{i}:[0,d(0,p_i)]\rightarrow X$ be a geodesic connecting $\0$ and $p_i$ with $\gamma_{i}(0)=\0$. 
    Notice that since $\Lambda H\subset\partial X^N_{\0}$, we have that $\gamma_i$ is $N'$-Morse for some $N'$ depending only on $N$. 
    By restricting to a subsequence, we may assume that 
	$\gamma_{i}$ converges, uniformly on compact subsets, to an $N'$-Morse geodesic ray $\gamma$ with $\gamma(0)=\0$.
  
	By construction and by Lemma \ref{WeakConvexHullAndLimitSetContainments}, $\gamma(\infty)\in \Lambda(WCH(\Lambda H))\subseteq\Lambda H$. So, by \cite[Corollary 2.6]{cordes_boundry2016}, there exists an $N$-Morse geodesic ray $\alpha$ with 
	$\alpha(0)=\0$ and $d(\alpha(t),\gamma(t))<D$ for all $t\geq 0$, where $D\geq0$ is
	a constant that depends only on $N$.  Let $T=2D+4$, and put $\alpha'=\alpha_{[T,\infty)}$. Since $\alpha'(\infty)\in \Lambda H$, and so by Theorem \ref{MorseHoroLimitsAreMorseFunnelLimits}, $\alpha'(\infty)$ is a funneled limit point of $H$. Thus there exists $h\in H$ so that $h\0\in F_{\0}^N(\alpha')$. Let $t_0=\min\{s:\alpha'(s)\in\pi_{\alpha'}(h\0)\}$. Since the sequence $\gamma_i$ converges uniformly on compact sets to $\gamma$, we may choose $i$ large enough so that $d(\gamma_i(t_0),\gamma(t_0))\leq 1$. See Figure \ref{FigureHoroLimitsImplieCobouned}.

 \begin{figure}[h]
\centering
\begin{tikzpicture}
        \draw[->, black] (0,2) to (10,2);
        \draw node[label=below:{$\alpha(\infty)$}] at (10,2) {};
        \draw node[label=below:{$\gamma(\infty)$}] at (10,4) {};
        \draw node[label=below:{$\0$},fill,circle,scale=0.5] at (0,2) {};
        \draw [-, black] (0,2)  to[out=90,in=180] (2,4);
        \draw [->, black] (2,4) to (10,4);
        \draw [-, black] (0,2) to[out=90, in=180] (2,4.5);
        \draw [-, black] (2,4.5) to (6,4.5);
        \draw [-, black] (6,4.5) to[out=0, in=-90] (8,6);
        \draw node[label=right:{$p_i$},fill,circle,scale=0.5] at (8,6) {};
        \draw node[label=above left:{$\gamma_i(t_0)$},fill,circle,scale=0.5] at (6,4.5) {};
        \draw node[label=below left:{$\gamma(t_0)$},fill,circle,scale=0.5] at (6,4) {};
        \draw node[label=above left:{$\alpha(t_0)$},fill,circle,scale=0.5] at (6,2) {};
        \draw node[label=left:{$h\0$},fill,circle,scale=0.5] at (6,0) {};
        \draw node[label=above:{$\alpha(T)$},fill,circle,scale=0.5] at (2,2) {};
        \draw node[label=below:{$2D+4$}] at (1,1.4) {};
        \draw [|-, black] (0,1) to (0.4,1);
        \draw [-|, black] (1.6,1) to (2,1);
        \draw [-, black] (6,0) to (6,4.5);
        \draw node[label=right:{$\leq 1$}] at (6,4.25) {};
        \draw node[label=right:{$\leq D$}] at (6,3) {};
        \draw node[label=right:{$\leq d(\alpha(T),\alpha(t_0))$}] at (6,1) {};
        \end{tikzpicture}\\
        \caption{Diagram for Theorem \ref{AllMorseHoroImpliesCoboundedAction}}\label{FigureHoroLimitsImplieCobouned}

\end{figure}
	 
	By the triangle inequality we have that $d(\gamma_i(t_0),\alpha'(t_0))\leq D+1$, and therefore $|d(0,\gamma_i(t_0))-d(0,\alpha(t_0))|\leq D+1$. Also, by construction we have that $d(h\0,\alpha'(t_0))\leq d(\alpha(T),\alpha(t_0))$. Therefore we have
    \begin{align*}
    d(p_i,h\0) &\leq d(p_i,\gamma_i(t_0))+d(\gamma_i(t_0),\alpha'(t_0))+d(\alpha('t_0),h_0)\\
    &\leq d(0,\gamma_i(t_0))+(D+1)+ d(\alpha(T),\alpha(t_0))\\
    &=d(\0,p_i)-d(\0,\gamma_i(t_0))+(D+1)+d(0,\alpha(t_0))-d(0,\alpha(T))\\
    &\leq (n_i+1)+(D+1)+(D+1)-(2D+4)\leq n_i-1.
    \end{align*}

    However, this contradicts the assumption that $B_{n_i}(p_i)\cap H\0=\emptyset$.
\end{proof}

We now present an alternate definition of a conical limit point which agrees with Definition \ref{conical limit point def} in the case where $\Lambda A$ is compact, and requires us to only consider of the geodesic rays which emanate from the given basepoint. By Corollary \ref{LimitSetsInstratumAreCompact} and by \cite[Lemma 4.1]{boundarycc2016}, the requirement that $\Lambda H$ is compact is equivalent to the requirement that $\Lambda H$ is contained in the boundary of a single Morse stratum. 

\begin{prop}\label{ConicalLimitPointsOnlyNeedTailsOfRays}
    Let $X$ be a proper, geodesic metric space. Let $Y\subseteq X$. Suppose $\Lambda Y\not=\emptyset$. Then the following are equivalent:
    
    \begin{enumerate}
        \item $x\in\partial X_{\0}$ is a conical limit point of $Y$ 
        \item There exists $K>0$ such that, for every $N$-Morse geodesic ray $\alpha:[0,\infty)\rightarrow X$ with $\alpha(0)=\0$ and $\alpha(\infty)=x$, and for every $T>0$, there exists $y\in Y$ such that $y\in \Nhbd_K(\alpha')$, where $\alpha':[0,\infty)\rightarrow X$ is defined by $\alpha'(t)=\alpha(t+T)$.
    \end{enumerate}
\end{prop}

\begin{proof}
    Showing that (1) implies (2) is a direct consequence of Lemma \ref{MorseSubraysAreMorse} and Definition \ref{conical limit point def}.
    
    Instead assume (2). Let $\beta:[b,\infty)\rightarrow X$ be an $N'$-Morse ray with $\beta(\infty)=x$. Let $\alpha:[0,\infty)\rightarrow X$ an $N$-Morse geodesic ray with $\alpha(0)=\0$ and $\alpha(\infty)=x$. Without loss of generality, by Cor \ref{betterswensonlemma4} there exists $T>0$ such that $d(\alpha(t),\beta(t))<\delta_N$ for all $t>T$. Put $\alpha':[0,\infty)\rightarrow X$ via $\alpha'(t)=\alpha(t+T)$. By hypothesis, there exists $y\in Y$ such that $y\in \Nhbd_K(\alpha')$. 
    Say $s\in[0,\infty)$ such that $d(\alpha'(s),y)<K$, so via the triangle inequality we have 
    \[d(\beta(s+T),y)\leq d(\beta(s+t),\alpha(s+t))+d(\alpha'(s),y)\leq \delta_N+K.\] Thus, $y\in \Nhbd_{K+\delta_N}(\beta)$, which shows (1).
\end{proof}

We conclude this section by showing that $(3) \Rightarrow (2)$ for Theorem \ref{MyMainTheorem}, which was first shown in \cite[Corollary 1.14]{boundarycc2016}, however here we present a direct proof that does not rely on \cite[Theorem 1.1]{boundarycc2016}.

\begin{prop}\label{BoundaryCCImpliesConical}
    Let $X$ be a proper geodesic metric space and let $H$ be a finitely generated group of isometries of $X$ such that the orbit map $H\rightarrow X$ via $h\mapsto h\0$ is a stable mapping. If there exists a Morse gauge $N$ so that $\Lambda H\subseteq \partial X_{\0}^N$, then every $x\in \Lambda H$ is a conical limit point of $H\0$.
\end{prop}

\begin{proof}
    Let $x\in\Lambda H$, and let $\alpha:[0,\infty)\rightarrow X$ be an $N$-Morse geodesic ray with $\alpha(\infty)=x$, $\alpha(0)=\0$. Let $\alpha'=\alpha|_{[a,\infty)}$ be a subray of $\alpha$. Notice that $\alpha'$ is $N'$-Morse where $N'$ depends only on $N$ by Lemma \ref{MorseSubraysAreMorse}. By Proposition \ref{ConicalLimitPointsOnlyNeedTailsOfRays}, it suffices to show that there exists some $K\geq 0$, depending only on $N'$ and $H$, so that $H\0\cap \Nhbd_K(\alpha')\not=\emptyset$.
    
    Since $H$ is a stable subgroup of isometries on $X$, we have that for any $h\in H$, there exists a $(\lambda,\lambda)$-quasi-geodesic $\gamma$ from $\0$ to $h\0$ such that, for any $p\in\gamma$, $B_{2\lambda}(p)\cap H\0\not=\emptyset$. (To find such a path $\gamma$, take a geodesic in a Cayley graph for $H$ and embed it into $X$ by extending the orbit map along appropriate geodesic segments.)
    
    Now, since $x\in\Lambda H$, there exists a sequence $h_n\in H$ such that the sequence of geodesic segments, $\beta_n=[\0,h_n\0]$, converges (uniformly on compact subsets) to a geodesic ray $\beta:[b,\infty)\rightarrow X$ with $\beta(\infty)=x$ and $\beta(b)=\0$. Since $H$ is a stable group of isometries, $\beta_n$ is $N''$-Morse by Definition \ref{stabledefinition}. Up to potentially re-parameterizing $\beta$, there exists $T>a$ so that $d(\beta(T),\alpha(T))<\delta_N$ by Corollary \ref{betterswensonlemma4}.
    
    Since $\beta_n$ converges to $\beta$ uniformly on $\overline{B_{T+1}(\0)}$, the ball of radius $T+1$ centered at $\0$, there exists $n\in\mathbb{N}$ and $p\in \beta_n$ so that $d(\beta(T),p)<1$. Since $\gamma_n$ is an $(\lambda,\lambda)$-quasi-geodesic with endpoints on $\beta_n$, there exists $q\in\gamma_n$ so that $d(p,q)\leq N''(\lambda,\lambda)$. Finally, there exists $h\in H$ so that $d(h\0,q)\leq\lambda$. 
    
    Therefore by the triangle inequality, \[d(\alpha(T),h\0)\leq d(\alpha(T),\beta(T))+d(\beta(T),p)+d(p,q)+d(q,h\0)\leq \delta_N+1+N''(\lambda,\lambda)+2\lambda.\] 
    As $\alpha(T)\in\alpha'$, this completes the proof.
\end{proof}


\section{Applictations to Teichm\"uller Space}

We conclude by illustrating applications of the above work in the setting of Teichm\"uller space for a finite type surface $S$. We begin by setting some notation. Let $\text{Mod}(S)$ denote the mapping class group of $S$, i.e. the group of orientation-preserving homeomorphisms on $S$ up to isotopy equivalence, which may permute punctures but fixes boundaries pointwise. Let $\mathcal{T}(S)$ denote the associated Teichm\"uller space, equipped with the Teichm\"uller metric. We will denote the set of projective measured foliations on $S$ by $\text{PMF}(S)$. The Thurston compactification of Teichm\"uller space is $\overline{\mathcal{T}(S)} = \mathcal{T}(S)\cup\text{PMF}(S)$. For a thorough overview of the mapping class group, it's associated Teichm\"uller space, and projective measured foliations, we refer the reader to \cite{MasurMinsky2000,fathi2012thurstons,Farb2011-rd,Behrstock2006}. 

We take a moment to restate Corollary \ref{TeichApplicationAllConical} using the above notation:

\begin{cor} (Restatement of Corollary \ref{TeichApplicationAllConical}.)
    Let $H$ be a finitely generated subgroup of $\textup{Mod}(S)$. The following are equivalent:

    \begin{enumerate}
        \item Every element of $\Lambda H\subset \partial \textup{Mod}(S)$ is a conical limit point of $H\curvearrowright\textup{Mod}(S)$ and $\Lambda H$ is compact (in the Morse boundary of $\textup{Mod}(S)$).
        \item Every element of $\Lambda H\subset \textup{PMF}(S)$ is a conical limit point of $H\curvearrowright\mathcal{T}(S)$.
    \end{enumerate}
\end{cor}

By work of Cordes, there exists a homeomorphism $g_\infty:\partial \text{Mod}(S)\rightarrow \partial \mathcal{T}(S)$ (where $\partial$ refers to the Morse boundary) \cite[Theorem 4.12]{cordes_boundry2016}, and there exists a natural continuous injective map $h_\infty: \partial\mathcal{T}(S)\hookrightarrow \text{PMF}(S)$ \cite[Proposition 4.14]{cordes_boundry2016}. We denote the continuous inclusion formed by the composition of $g_\infty$ and $h_\infty$ as $f_\infty: \partial\textup{Mod}(S)\hookrightarrow \text{PMF}(S)$. The purpose of this section is to prove the following theorem.

\begin{thm}\label{TeichApplicationOneConical}
    Let $H$ be a subgroup of $\textup{Mod}(S)$, and let $x_\infty\in\Lambda H\subseteq \partial\textup{Mod}(S)$ be a conical limit point of $H\curvearrowright\textup{Mod}(S)$. Then $f_\infty(x_\infty)\in\textup{PMF}(S)$ is a conical limit point of $H\curvearrowright\mathcal{T}(S)$.
\end{thm}

\begin{remark}
    This theorem directly proves $(1)\Rightarrow(2)$ of Corollary \ref{TeichApplicationAllConical}.
\end{remark}

Our proof of Theorem \ref{TeichApplicationOneConical} uses several of the tools developed in \cite{cordes_boundry2016}, so we take a moment to recall the construction and definitions presented therein and from \cite{MasurMinsky2000}. The \textit{curve graph}, denoted $\mathcal{C}(S)$, is a locally infinite simplicial graph whose vertices are isotopy classes of simple closed curves on $S$. We join two vertices with an edge it there exists representative from each class that are disjoint. 

A set of (pairs of) curves $\mu=\{(\alpha_1,\beta_1),(\alpha_2,\beta_2),\dots,(\alpha_m,\beta_m)\}$ is called a complete clean \textit{marking} of $S$ if the $\{\alpha_1,\dots,\alpha_m\}$ forms a pants decomposition of $S$, if each $\alpha_i$ is disjoint from $\beta_j$ whenever $i\not=j$, and if each $\alpha_i$ intersects $\beta_i$ once if the surface filled by $\alpha_i$ and $\beta_i$ is a one-punctured torus. (Otherwise, $\alpha_i$ and $\beta_i$ will intersect twice, and the filling surface is a four-punctured sphere.) We call $\{\alpha_1,\dots,\alpha_m\}$ the \textit{base} of $\mu$ and we call $\beta_i$ the \textit{transverse curve to $\alpha_i$ in $\mu$}. For the sake of completeness, we also define the marking graph, $\mathcal{M}(S)$, although the definition is not needed in this paper. 

$\mathcal{M}(S)$ is the simplicial graph whose vertices are markings as defined above, and two markings are joined by an edge of length one if they differ by an \emph{elementary move}: either twisting $\beta_i$ around $\alpha_i$ by a full, or when possible, by a half twist, or by swapping $\beta_i$ and $\alpha_i$. Note that, after performing an elementary move, one may need to replace the curves with isotopically equivalent curves to create a valid marking again. The marking graph $\mathcal{M}(S)$ is quasi-isometric to the mapping class group $\textup{Mod}(S)$, see \cite{MasurMinsky2000} and \cite{Behrstock2006}.

For each $\sigma\in\mathcal{T}(S)$ there is a \textit{short marking}, which is constructed inductively by picking the shortest curves in $\sigma$ for the base and repeating for the transverse curves. Now define a map $\Upsilon:\mathcal{M}(S)\rightarrow\mathcal{T}(S)$ by taking a marking $\mu$ to the region in the $\epsilon$-thick part of $\mathcal{T}(S)$, denoted $\mathcal{T}_\epsilon(S)$, where $\mu$ is a short marking in that region. As stated in \cite{cordes_boundry2016}, it is a well known fact that $\Upsilon$ is a coarsely well defined map which is coarsely Lipschitz. We take a moment to prove that this map is coarsely equivariant. 

\begin{lemma}\label{UpsilonIsCoarseEquivariant}
    Let $\Upsilon:\mathcal{M}(S)\rightarrow\mathcal{T}(S)$ be as above, and let $H<\textup{Mod}(S)$ be finitely generated. Then there exists a constant $K\geq0$ such that, for any marking $\mu\in\mathcal{M}$ and for any $h\in H$,
    $$d_{\mathcal{T}(S)}(h\Upsilon(\mu),\Upsilon(h\mu))\leq K.$$
\end{lemma}

\begin{proof}
    Let $\mu=\{(\alpha_1,\beta_1),\dots,(\alpha_m,\beta_m)\}\in\mathcal{M}(S)$ and $h\in H$ be arbitrary. Let $\sigma\in\mathcal{T}(S)$ so that $\mu$ is a short marking on $\sigma$. (Equivalently, let $\sigma=\Upsilon(\mu)$.) Since the action of $H$ on $\mathcal{T}(S)$ permutes the lengths of curves, the length of each pair $(\alpha_i,\beta_i)$ with respect to $\sigma$ is the same as the length of the pair $(h\alpha_i,h\beta_i)$ with respect to $h\sigma$. Therefore as $\mu$ was a short marking for $\sigma$, this shows that $h \mu$ is a short marking for $h\sigma=h\Upsilon(\mu)$. However, by definition of $\Upsilon$, $h\mu$ is also a short marking for $\Upsilon(h\mu)$. As $\Upsilon$ was a coarsely well defined function, this shows that $d_{\mathcal{T}(S)}(h\Upsilon(\mu),\Upsilon(h\mu))\leq K$ for some $K\geq 0$, as desired.
\end{proof}

We now prove Theorem \ref{TeichApplicationOneConical}, using the above lemma and several tools from \cite{cordes_boundry2016} to show that points in conical neighborhoods in $\mathcal{M}(S)$ end up in conical neighborhoods of $\mathcal{T}(S)$.

\begin{proof}
    Fix $\mu_0\in\mathcal{M}(S)$. Let $x\in\partial\mathcal{M}(S)_{\mu_0}$ be a conical limit point of $H\mu_0$. Put $\sigma_0=\Upsilon(\mu_0)$. We shall show that $f_\infty(x)$ is a conical limit point of $H\sigma_0$ by verifying the condition in Proposition \ref{ConicalLimitPointsOnlyNeedTailsOfRays}. Let $T\geq 0$ be arbitrary, and let $\lambda:[0,\infty)\rightarrow\mathcal{T}(S)$ be an arbitrary Morse geodesic ray with $\lambda(0)=\sigma_0$ and $\lambda(\infty)=f_\infty(x)$. 
    
    Let $\alpha:\mathbb{N}\rightarrow\mathcal{M}(S)$ be an $N$-Morse geodesic with $\alpha(0)=\mu_0$ and $\alpha(\infty)=x$. By \cite[Lemma 4.9]{cordes_boundry2016}, $\Upsilon(\alpha)$ is an $N'$-Morse $(A,B)$-quasi-geodesic, for some $A$, $B$, and $N'$ depending only on $N$. Put $\beta=\Upsilon(\alpha)$. Notice that $\beta(0)=\sigma_0$ and, by the construction of $f_\infty$, we have $\beta(\infty)=f_\infty(x)$. (For details on the construction of $f_\infty$, we refer to \cite{cordes_boundry2016}, specifically Proposition 4.11, Theorem 4.12, and Proposition 4.14.)

    Now let $\gamma_n = [\sigma_0, \beta(n)]$. Then each $\gamma_n$ is $N''$-Morse for $N''$ depending on $N$, and by Arzel\'a-Ascoli and \cite[Lemma 2.10]{cordes_boundry2016}, a subsequence of the $\gamma_n$ converges to a geodesic ray $\beta$ which is $N''$-Morse, and by \cite[Lemma 4.9]{cordes_boundry2016}, $\beta$ is bounded Hausdorff distance from $\gamma$, where the bound only depends on $N$. Say that $d_{Haus}(\beta,\gamma)\leq K_1$ for $K_1\geq0$. By \cite[Corollary 2.6]{cordes_boundry2016}, $d_{Haus}(\gamma,\lambda)\leq K_2$ where $K_2\geq0$ depends only on $N$.
    Choose $S\geq0$ so that, for all $s\geq S$, $d_{\mathcal{T}(S)}(\beta(s),\lambda_{[T,\infty)})\leq K_1+K_2$.

    By Proposition \ref{ConicalLimitPointsOnlyNeedTailsOfRays}, there exists $L\geq 0$ where, for all $r\geq0$, $d_{\mathcal{M}(S)}(h\mu_0,\alpha|_{[r,\infty)})\leq L$ for some $h\in H$. Since $\beta=\Upsilon(\alpha)$ and $\Upsilon$ is coarse Lipschitz, there exits $K_3\geq0$ and $h\in H$ so that $d_{\mathcal{T}(S)}(\Upsilon(h\mu_0),\beta|_{[S,\infty)})\leq K_3$. Let $s_0\in[S,\infty)$ so that $d_{\mathcal{T}(S)}(\Upsilon(h\mu_0),\beta(s_0))\leq K_3$. By Lemma \ref{UpsilonIsCoarseEquivariant}, there exists $K_4\geq0$ such that $d_{\mathcal{T}(S)}(\Upsilon(h\mu_0),h\Upsilon(\mu_0))\leq K_4$.

    By the triangle inequality, we have
    \begin{align*}
    d_{\mathcal{T}(S)}(h\sigma_0,\lambda|_{[T,\infty)}) &\leq d(h\Upsilon(\mu_0),\Upsilon(h\mu_0)) +d(\Upsilon(h\mu_0),\beta(s_0)) +d(\beta(s_0),\lambda|_{[T,\infty)})\\ &\leq K_4+K_3+K_2+K_1
    \end{align*}

    By Proposition \ref{ConicalLimitPointsOnlyNeedTailsOfRays}, $\lambda(\infty)=f_\infty(x)$ is a conical limit point of $H\sigma_0$.
\end{proof}

\printbibliography

\end{document}